\documentclass[12pt]{amsart}

\usepackage{times}
\usepackage{amssymb}
\usepackage{enumitem}
\usepackage{gensymb}

\usepackage{color}
\usepackage[usenames,dvipsnames,svgnames,table]{xcolor}

\usepackage{tikz}
\usetikzlibrary{matrix,arrows,calc}
\usetikzlibrary{decorations.markings}
\usetikzlibrary{patterns}
\usetikzlibrary{snakes}
\usetikzlibrary{shapes}

\usepackage{graphicx}

\usepackage[config, labelfont={normalsize}]{caption,subfig}
\usepackage{array}

\usepackage[T1]{fontenc}
\usepackage[utf8]{inputenc}
\usepackage{bm}

\usepackage{hyperref}
\usepackage{url}
\usepackage[capitalise,noabbrev]{cleveref}

\theoremstyle{plain}
\newtheorem{lemma}{Lemma}[section]
\newtheorem{theorem}[lemma]{Theorem}
\newtheorem{corollary}[lemma]{Corollary}
\newtheorem{proposition}[lemma]{Proposition}

\newtheorem{conjecture}[lemma]{Conjecture}
\newtheorem{themainconjecture}[lemma]{The Main Conjecture}
\newtheorem{question}[lemma]{Question}

\theoremstyle{remark}
\newtheorem{remark}[lemma]{Remark}
\newtheorem{example}[lemma]{Example}
\newtheorem{definition}[lemma]{Definition}

\DeclareMathOperator{\weight}{mon}
\DeclareMathOperator{\vagueweight}{wt}

\DeclareMathOperator{\contribution}{contribution}
\DeclareMathOperator{\Ch}{Ch}
\DeclareMathOperator{\id}{id}

\DeclareMathOperator{\Tr}{Tr}

\newcommand{\Moeb}{d}

\newcommand{\R}{\mathbb{R}}

\newcommand{\C}{\mathbb{C}}
\newcommand{\N}{\mathbb{N}}

\newcommand{\HH}{\mathcal{H}}

\newcommand{\QQ}{\mathbb{Q}}

\newcommand{\Sym}[1]{\mathfrak{S}_{#1}}

\newcommand{\newCh}{\widehat{\Ch}}

\newcommand{\V}{\mathcal{V}}
\newcommand{\F}{\mathcal{F}}
\newcommand{\loops}{\mathcal{L}}

\newcommand{\WC}{\mathcal{W}}
\newcommand{\BC}{\mathcal{B}}
\newcommand{\E}{\mathcal{E}}

\date{}

\author[M.~Dołęga]{Maciej Dołęga}
 \address{LIAFA, Universit\'e Paris 7, Case 7014, 75205 Paris Cedex 13,
France \newline \indent Instytut Matematyczny,
Uniwersytet Wrocławski,  \mbox{pl.\ Grunwaldzki~2/4,} 50-384
Wrocław, Poland}
\email{dolega@liafa.univ-paris-diderot.fr}

\author[V.~F\'eray]{Valentin F\'eray}
\address{LaBRI, Universit\'e Bordeaux 1, 351 cours de la Lib\'eration, 33 400
Talence, France}
\email{feray@labri.fr}

\author [P.~Śniady]{Piotr \'Sniady}
\address{
Wydział Matematyki i Informatyki, 
Uniwersytet im.~Adama Mickiewicza, 
Collegium Mathematicum,
Umultowska 87, 
61-614 Poznań, 
Poland, \newline \indent
Instytut Matematyczny, Polska Akademia Nauk, 
\mbox{ul.~\'Sniadec\-kich 8,} \linebreak 00-956 Warszawa, Poland
} 
\email{piotr.sniady@amu.edu.pl}

\title[Jack polynomials and orientability generating series of maps]{Jack polynomials \\ and orientability generating series of maps}

\begin{document}

\begin{abstract}
We study \emph{Jack characters}, 
which are the coefficients of the power-sum expansion of Jack symmetric functions with a suitable normalization. 
These quantities have been introduced by Lassalle who formulated some challenging conjectures about them. 
We conjecture existence of a weight on non-oriented 
maps (i.e., graphs drawn on non-oriented surfaces)
which allows to express 
any given Jack character as a weighted sum of some simple functions indexed by maps. 
We provide a candidate for this weight which gives a positive answer to our conjecture in some, 
but unfortunately not all, cases. 
In particular, it gives a positive answer for Jack characters specialized on Young diagrams of rectangular shape. 
This candidate weight attempts to measure,
in a sense, the non-orientability of a given map.
\end{abstract}



\maketitle

\numberwithin{equation}{section}
\numberwithin{figure}{section}

\section{Introduction}

\subsection{Jack polynomials and Macdonald polynomials}
Jack \cite{Jack1970/1971}
introduced a family of symmetric polynomials ---
which are now known as \emph{Jack polynomials} $J^{(\alpha)}_\pi$ --- indexed by a partition 
and a deformation parameter $\alpha$.
From the contemporary point of view, probably the main motivation for studying Jack polynomials 
comes from the fact that they are a special case of the celebrated \emph{Macdonald polynomials} which 
\emph{``have found applications in special function theory, representation theory, algebraic geometry, 
group theory, statistics and quantum mechanics''} \cite{GarsiaRemmel2005}. 
Indeed, some surprising features of Jack polynomials \cite{Stanley1989} have led in the past to the discovery of 
Macdonald polynomials \cite{Macdonald1995} and Jack polynomials have been regarded as a relatively easy 
case \cite{LapointeVinet1995} which later allowed understanding of the more difficult case of 
Macdonald polynomials \cite{LapointeVinet1997}.
A brief overview of Macdonald polynomials (and their relationship to 
Jack polynomials) is given in \cite{GarsiaRemmel2005}.

Jack polynomials are also interesting on their own,
for instance in the context of Selberg integrals \cite{Kadell1997}
and in theoretical physics \cite{Feigin2002,PhysRevLett.100.246802}.

\subsection{Jack polynomials, Schur polynomials and zonal polynomials}
For some special choices of the deformation parameter $\alpha$, Jack polynomials coincide 
(up to some simple normalization constants) with some very established families of symmetric polynomials. 
In particular, the case $\alpha=1$ corresponds to \emph{Schur polynomials}, 
$\alpha=2$ corresponds to \emph{zonal polynomials}, 
and $\alpha=\frac{1}{2}$ corresponds to \emph{symplectic zonal polynomials}; 
see \cite[Chapter 1 and Chapter 7]{Macdonald1995} for more information about these functions.
For these special values of the deformation parameter, 
Jack polynomials are particularly nice because they have some additional structures and features 
(usually related to algebra and representation theory) and for this reason they are much better understood. 

\subsection{Jack polynomials and maps}

\begin{figure}
\centering
\begin{tikzpicture}[scale=0.6,
white/.style={circle,draw=black,fill=white,inner sep=4pt},
black/.style={circle,draw=black,fill=black,inner sep=4pt},
connection/.style={draw=black!80,black!80,auto}
]
\footnotesize

\begin{scope}
\clip (0,0) rectangle (10,10);
\fill[blue!3] (0,0) rectangle (10,10);
\draw (3,5) node (b1) [black] {};
\draw (b1) +(10,0) node (b1prim) [black] {};

\draw (b1) +(1,-3) node (b1-se) [white] {};
\draw (b1) +(-1,-3) node (b1-sw) [white] {};

\draw (8,8) node (b2) [black] {};
\draw (b2) +(0,-10) node (b2prim) [black] {};
\draw (b2) +(-10,0) node (b2prim2) [black] {};

\draw (6,5) node (w1) [white] {};
\draw (w1) +(0,10) node (w1prim) {};

\draw (12,7) node (w2) [white] {};
\draw (w2) +(-10,0) node (w2prim) [white] {};

\draw[connection,pos=0.2] (b2) to node {\textcolor{black}{$4$}} node [swap] {} (w1prim);
\draw[connection,pos=0.666] (b2prim) to node {\textcolor{black}{$4$}} node [swap] {} (w1);

\draw[connection,pos=0.2] (b2) to node {\textcolor{black}{$6$}} node [swap] {} (w2);
\draw[connection,pos=0.7] (b2prim2) to node {\textcolor{black}{$6$}} node [swap] {} (w2prim);

\draw[connection] (b1) to node {\textcolor{black}{$2$}} node [swap] {} (b1-sw);

\draw[connection] (b1) to node {\textcolor{black}{$3$}} node [swap] {} (b1-se);

\draw[connection] (b1) to node {\textcolor{black}{$5$}} node [swap] {} (w1);

\draw[connection] (w1) to node {} node [swap] {\textcolor{black}{$7$}} (b2);

\draw[connection] (w2) to node {} node [swap] {\textcolor{black}{$1$}} (b1prim);
\draw[connection] (w2prim) to node {} node [swap] {\textcolor{black}{$1$}} (b1);

\end{scope}

\draw[very thick,decoration={
    markings,
    mark=at position 0.666  with {\arrow{>}}},
    postaction={decorate}]  
(0,0) -- (10,0);

\draw[very thick,decoration={
    markings,
    mark=at position 0.666  with {\arrow{>}}},
    postaction={decorate}]  
(0,10) -- (10,10);

\draw[very thick,decoration={
    markings,
    mark=at position 0.666  with {\arrow{>>}}},
    postaction={decorate}]  
(0,0) -- (0,10);

\draw[very thick,decoration={
    markings,
    mark=at position 0.666  with {\arrow{>>}}},
    postaction={decorate}]  
(10,0) -- (10,10);

\end{tikzpicture}
\caption{Example of an \emph{oriented map}. The map is drawn on a torus: the left side of the square should be glued to the right side,
as well as bottom to top, as indicated by the arrows. }
\label{fig:map-kerov}
\end{figure}
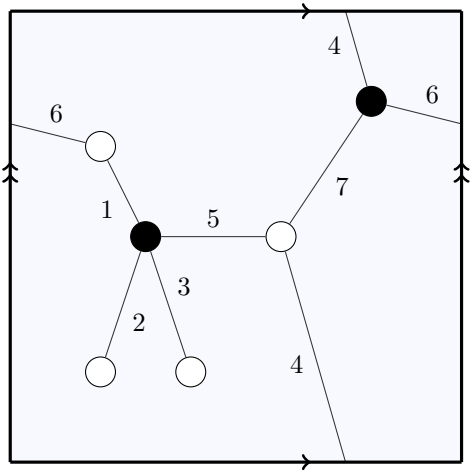

Roughly speaking, a \emph{map} is a graph drawn on a surface, see \cref{fig:map-kerov}.
In this article we will investigate the relationship between combinatorics of Jack polynomials and enumeration of maps.

In the special cases of Schur polynomials ($\alpha=1$) and zonal polynomials ($\alpha=2$ and $\alpha=\frac{1}{2}$) 
this relationship is already well-understood. 
The generic case is much more mysterious and we will be able only to present some partial results.

\subsection{Normalized characters}
The \emph{irreducible character} $\chi^\lambda(\pi)$ of the symmetric group is usually considered as a function 
of the partition $\pi$, with the Young diagram $\lambda$ fixed. It was a brilliant observation of Kerov and Olshanski
\cite{KerovOlshanski1994} that for several problems in the asymptotic representation theory it is convenient 
to do the opposite: keep the partition $\pi$ fixed and let the Young diagram $\lambda$ vary.  
It should be stressed that in this approach the Young diagram $\lambda$ is arbitrary, 
in particular there are no restrictions on the number of boxes of $\lambda$. 
In this way it is possible to study the structure of the series of the 
symmetric groups $\Sym{1}\subset \Sym{2}\subset \cdots$ and their representations in a uniform way. 
This concept is sometimes referred to as \emph{dual approach} to the characters of the symmetric groups.

In order for this idea to be successful, one has to replace the usual characters $\chi^\lambda(\pi)$ 
by the \emph{normalized characters} $\Ch_\pi(\lambda)$. 
Namely, for a partition $\pi$ of size $k$ and a Young diagram $\lambda$ with $n$ boxes we define 
\begin{equation}
\Ch_\pi(\lambda) =\begin{cases} (n)_k\ \frac{\chi^\lambda(\pi \cup 1^{n-k})}{\chi^\lambda(\id)} &
\text{if } k\leq n, \\ 0 & \text{otherwise,} \end{cases}
\label{eq:definition-character} 
\end{equation}
where$\chi^{\lambda}(\pi \cup 1^{n-k})$ is the character of the irreducible
representation indexed by $\lambda$ evaluated on a permutation
of cycle-type $\pi \cup 1^{n-k}$ and
\[(n)_k:=\underbrace{n (n-1) \cdots (n-k+1)}_{\text{$k$ factors}}\] 
denotes the \emph{falling factorial}.

This choice of normalization is justified by the fact that the so defined characters $\Ch_\pi$ belong to the algebra of \emph{polynomial functions on the set of Young diagrams} \cite{KerovOlshanski1994}, which in the last two decades turned out to be essential for several asymptotic and enumerative problems of the representation theory of the symmetric groups 
\cite{Biane1998,IvanovOlshanski2002,DolegaF'eray'Sniady2008}.

\subsection{Jack characters}
Lassalle \cite{Lassalle2008a,Lassalle2009} initiated investigation of a kind of \emph{dual approach} to Jack polynomials.
Roughly speaking, it is the investigation (as a function of $\lambda$, with $\pi$ being fixed) of 
the coefficient standing at $p_{\pi,1,1,\dots,1}$ in the expansion of the Jack symmetric polynomial 
$J^{(\alpha)}_\lambda$ in the basis of \emph{power-sum symmetric functions}. 
This coefficient, with some appropriate normalization factor given in \cref{sec:JackCh},
will be denoted by $\Ch_\pi^{(\alpha)}(\lambda)$.

This normalization factor is chosen in such a way that
in the important special case of the Schur polynomials ($\alpha=1$)
one recovers the normalized characters $\Ch_\pi^{(1)}=\Ch_\pi$ 
given by \eqref{eq:definition-character} which already proved to have a rich and fascinating structure. 
The cases of zonal polynomials, symplectic zonal polynomials and general Jack polynomials give rise to some 
new quantities for which in \cite{FeraySniady2011} we coined the 
names \emph{zonal characters} $\Ch_\pi^{(2)}$, \emph{symplectic zonal characters} $\Ch_\pi^{(1/2)}$, 
and general \emph{Jack characters} $\Ch_\pi^{(\alpha)}$. 

Another motivation for studying Jack characters $\Ch_\pi^{(\alpha)}$ comes from the observation that they form a linear basis of the algebra of \emph{$\alpha$-polynomial functions on the set of Young diagrams} (which is a simple deformation of the algebra of polynomial functions mentioned above). This fact is far from being trivial and was established by Lassalle 
\cite[Proposition 2]{Lassalle2008a}.

The main goal of this paper is to \emph{understand the combinatorial structure of Jack characters $\Ch_\pi^{(\alpha)}$}. In the following we will give more details about this problem.

\subsection{Embeddings of bicolored graphs}
\label{sec:SubsectStanleyGeneral}

\newcommand{\kolorSigma}{blue}
\newcommand{\kolorPi}{red}
\newcommand{\kolorV}{DarkGreen}
\newcommand{\kolorW}{RawSienna}

\begin{figure}
\centering
\subfloat[][]{\begin{tikzpicture}[scale=0.5,
white/.style={circle,draw=black,fill=white,inner sep=4pt},
black/.style={circle,draw=black,fill=black,inner sep=4pt},
connection/.style={draw=black!80,black!80,auto}
]
\footnotesize

\begin{scope}
\clip (0,0) rectangle (10,10);
\fill[blue!3] (0,0) rectangle (10,10);

\draw (3.333,2.333) node (b1)    [black,label=110:$\textcolor{\kolorPi}{\Pi}$] {};
\draw (b1) +(10,0) node (b1prim) [black] {};

\draw (7.666,6.666) node (b2)     [black,label=0:\textcolor{\kolorSigma}{$\Sigma$}] {};
\draw (b2) +(0,-10) node (b2prim) [black] {};

\draw (b2) +(-3,1) node (w2) [white,label=180:$\textcolor{\kolorW}{W}$] {};

\draw (6.666,3.333) node (w1) [white,label=120:$\textcolor{\kolorV}{V}$] {};
\draw (w1) +(-10,0) node (w1left) [white] {};
\draw (w1) +(0,10)  node (w1top)  [white] {};

\draw[connection] (b1) to node {\textcolor{black}{$4$}} node [swap] {} (w1);

\draw[connection] (b2) to node {\textcolor{black}{$3$}} node [swap] {} (w2);

\draw[connection,pos=0.333] (b2) to node {\textcolor{black}{$2$}} node [swap] {} (w1top);
\draw[connection,pos=0.8] (b2prim) to node {\textcolor{black}{$2$}} node [swap] {} (w1);

\draw[connection,pos=0.666] (b1prim) to node {\textcolor{black}{$1$}} node [swap] {} (w1);
\draw[connection,pos=0.333] (b1) to node {\textcolor{black}{$1$}} node [swap] {} (w1left);

\draw[connection,pos=0.666] (w1) to node {} node [swap] {\textcolor{black}{$5$}} (b2);

\end{scope}

\draw[very thick,decoration={
    markings,
    mark=at position 0.666  with {\arrow{>}}},
    postaction={decorate}]  
(0,0) -- (10,0);

\draw[very thick,decoration={
    markings,
    mark=at position 0.666  with {\arrow{>}}},
    postaction={decorate}]  
(0,10) -- (10,10);

\draw[very thick,decoration={
    markings,
    mark=at position 0.666  with {\arrow{>>}}},
    postaction={decorate}]  
(0,0) -- (0,10);

\draw[very thick,decoration={
    markings,
    mark=at position 0.666  with {\arrow{>>}}},
    postaction={decorate}]  
(10,0) -- (10,10);
\end{tikzpicture}
\label{subfig:map}}
\hfill
\subfloat[][]{
\begin{tikzpicture}[scale=1.2]
\begin{scope}

\draw[line width=5pt,\kolorSigma!20] (-0.2,0.5) -- (3.2,0.5);
\draw (3.2,0.5) node[anchor=west] {$\textcolor{\kolorSigma}{\Sigma}$};

\draw[line width=5pt,\kolorPi!20] (-0.2,1.5) -- (1.2,1.5);
\draw (1.2,1.5) node[anchor=west] {$\textcolor{\kolorPi}{\Pi}$};

\draw[line width=5pt,\kolorW!20] (2.5,-0.2) -- (2.5,1.2);
\draw (2.5,1.2) node[anchor=south] {$\textcolor{\kolorW}{W}$};

\draw[line width=5pt,\kolorV!20] (0.5,-0.2) -- (0.5,2.2);
\draw (0.5,2.2) node[anchor=south] {$\textcolor{\kolorV}{V}$};

\draw[ultra thick] (4.5,0) -- (3,0) -- (3,1) -- (1,1) -- (1,2) -- (0,2) -- (0,2.5); 
\draw (0,0) -- (3,0) -- (3,1) -- (1,1) -- (1,2) -- (0,2); 
\clip (0,0) -- (3,0) -- (3,1) -- (1,1) -- (1,2) -- (0,2); 
\draw (0,0) grid (3,3);
\end{scope}
\draw (0.5,-0.2) node[anchor=north,text height=8pt] {$a$};
\draw (1.5,-0.2) node[anchor=north,text height=8pt] {$b$};
\draw (2.5,-0.2) node[anchor=north,text height=8pt] {$c$};

\draw (-0.2,0.5) node[anchor=east]  {$\alpha$};
\draw (-0.2,1.5) node[anchor=east] {$\beta$};

\draw(2.5,0.5) node {$3$};
\draw(0.5,0.5) node {$2,5$};
\draw(0.5,1.5) node {$1,4$};

\end{tikzpicture}
\label{subfig:embed}}

\caption{\protect\subref{subfig:map} Example of a bicolored graph (drawn on the torus) and 
\protect\subref{subfig:embed}  an example of its embedding $F$ into the Young diagram $(3,1)$;
this embedding is given by 
$F(\Sigma)=\alpha$, $F(\Pi)=\beta$, $F(V)=a$, $F(W)=c$,
$F(1)=F(4)=(a \beta)$, $F(2)=F(5)=(a \alpha)$, $F(3)=(c\alpha)$.
The white vertices of the graph were labeled by capital Latin letters, 
the black vertices by capital Greek letters,
and the edges by Arabic numbers.
The columns of the Young diagram were indexed by small Latin letters, the rows by small Greek letters.}
\label{fig:embedding}
\end{figure}
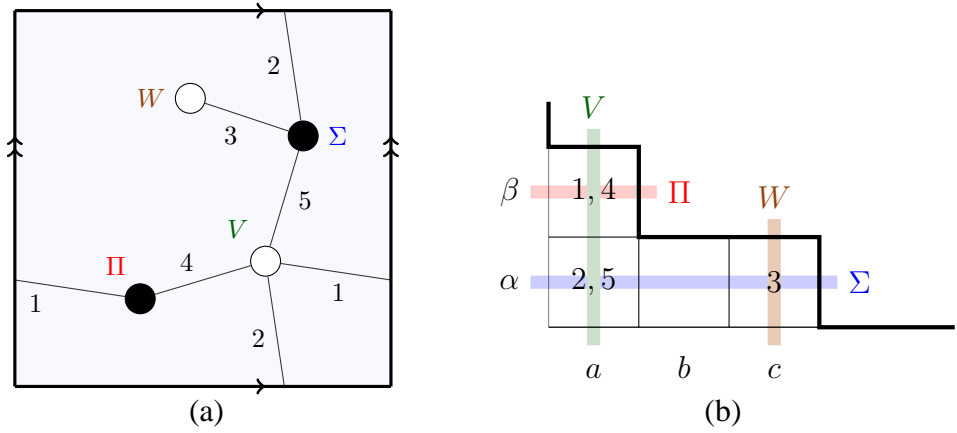

A \emph{bicolored graph $G$} is defined as a bipartite graph together with the choice of the 
coloring of its vertex set $\V=\V(G)$; we denote by
$\V_\bullet=\V_\bullet(G)$ and $\V_\circ=\V_\circ(G)$ respectively the sets
of black and white vertices of $G$.

An \emph{embedding $F$ of a bicolored graph $G$ to a Young diagram $\lambda$} is a function
which maps $\V_\circ$ to the set of columns of $\lambda$, 
maps $\V_\bullet$ to the set of rows of $\lambda$, 
and maps the edge set $\E(G)$ of $G$ to the set of boxes of $\lambda$, see \cref{fig:embedding}. 
We also require that an embedding preserves the relation of \emph{incidence},
i.e.,~a vertex $v\in \V(G)$ and an incident edge $e\in \E(G)$ 
should be mapped to a row or column $F(v)$ which contains the box
$F(e)$. 

We denote by $N_G(\lambda)$ the number of such embeddings of $G$ to $\lambda$. 
The quantities $N_G(\lambda)$ were introduced in the paper \cite{FeraySniady2011a} 
and they proved to be very useful for studying various asymptotic and enumerative problems 
of the representation theory of the symmetric groups \cite{FeraySniady2011a,DolegaF'eray'Sniady2008}.

When dealing with Jack polynomials, it will be convenient to consider
the following slightly deformed version:
\begin{equation}
\label{eq:ng-alpha}
N_G^{(\alpha)}(\lambda):= \left( \frac{1}{\sqrt{\alpha}}\right)^{|\V_\bullet(G)|}
             \left(  {\sqrt{\alpha}}\right)^{|\V_\circ(G)|} N_G(\lambda).   
\end{equation}
Note that $N_G^{(1)}(\lambda)=N_G(\lambda)$.

\subsection{Stanley formulas}
We will recall some formulas that express 
Jack characters $\Ch^{(\alpha)}_\pi$ in terms of functions $N^{(\alpha)}_G$
in the special cases $\alpha\in\left\{\frac{1}{2},1,2\right\}$.
Formulas of this type are called \emph{Stanley formulas} after
Stanley, who found such a formula for $\alpha=1$
as a conjecture \cite{Stanley-preprint2006}.

\subsubsection{Stanley formula for $\alpha=1$ and oriented maps}

Roughly speaking, an \emph{oriented map $M$} is defined as a bicolored graph 
drawn on an \emph{oriented} surface, see \cref{fig:map-kerov}. 
Since oriented maps are not in the focus of this article, 
we do not present all necessary definitions and for the details we refer to \cite{FeraySniady2011a}. 

It has been observed in \cite{FeraySniady2011a} that a certain formula 
conjectured by Stanley \cite{Stanley-preprint2006}
and proved by the second author \cite{F'eray2010}
for the normalized characters of the symmetric groups 
can be expressed as the sum
\begin{equation}
\label{eq:alpha=1}
 \Ch^{(1)}_\pi(\lambda)= \Ch_\pi(\lambda) =(-1)^{\ell(\pi)}\ \sum_{M} (-1)^{|\V_\bullet(M)|}\ 
 N^{(1)}_M(\lambda) 
\end{equation}
over all \emph{oriented bicolored maps $M$ with face-type $\pi$}. 
Here and throughout the paper, $N^{(\alpha)}_M$ denotes the function $N^{(\alpha)}$
indexed by the underlying graph of the map $M$ while
$\ell(\pi)$ denotes the number of parts of the partition $\pi$.

\subsubsection{Stanley formula for $\alpha\in\{2,\frac{1}{2}\}$ and non-oriented maps}
Roughly speaking, a \emph{non-oriented map} --- or, shortly, \emph{map} --- 
is a bicolored graph drawn on a surface. 
For a precise definition (and the definition of the face-type) 
we refer to \cref{subsec:non-oriented}.

In \cite{FeraySniady2011} it has been proved that
\begin{align} 
\label{eq:alpha=2}
\Ch_\pi^{(2)} (\lambda) &= (-1)^{\ell(\pi)}\ \sum_M  (-1)^{|\V_\bullet(M)|}
 \cdot \left(- \frac{1}{\sqrt{2}} \right)^{|\pi|+\ell(\pi)-|\V(M)|}
  N^{(2)}_M(\lambda),\\
\label{eq:alpha=1/2}
\Ch_\pi^{(1/2)} (\lambda) &= (-1)^{\ell(\pi)}\ \sum_M  \left( 
-1\right)^{|\V_\bullet(M)|}
  \cdot \left(\frac{1}{\sqrt{2}} \right)^{|\pi|+\ell(\pi)-|\V(M)|}
  N^{(1/2)}_M(\lambda), 
\end{align}
where the sums run over all \emph{non-oriented maps $M$
with the face-type $\pi$}.

The notations and the normalizations in \cite{FeraySniady2011}
are a bit different, so in \cref{SubsectReformulationFS11}
we make the link between the statements above and the results of \cite{FeraySniady2011}.

\subsection{The main conjecture}
\label{sec:MainConjecture}
Based on the special cases above, on the theoretical results of this paper,
and some computer exploration, we dare to formulate the following conjecture.
\begin{themainconjecture}
    \label{conj:main}
    To each non-oriented map $M$ one can associate some weight $\vagueweight_M(\gamma)$
    such that:
    \begin{itemize}
        \item for every $\lambda$ and $\pi$, the following formula holds:
\begin{equation}
\label{eq:conjectural-mysterious-formula}
\Ch_\pi^{(\alpha)} (\lambda) =  (-1)^{\ell(\pi)} 
            \sum_M  \left(-1\right)^{|\V_\bullet(G)|}
          \vagueweight_M \left( \frac{1-\alpha}{\sqrt{\alpha}}\right)
          N^{(\alpha)}_M(\lambda),
\end{equation}
where the sum runs over all \emph{non-oriented maps $M$ with the face-type $\pi$};
        \item $\vagueweight_M(\gamma)$ is a polynomial in variable $\gamma$ 
with non-negative rational coefficients, of degree (at most)
\[
 \Moeb(M):=2 (\text{number of connected components of $M$}) - \chi(M), 
\]
where \[\chi(M):=|\V(M)|-|\E(M)|+|\F(M)|\]
is the Euler characteristic of $M$. Moreover, the polynomial $\vagueweight_M(\gamma)$ is an even (respectively, odd) polynomial if and only if
the Euler characteristic $\chi(M)$ is an even number (respectively, an odd number).
    \end{itemize}
\end{themainconjecture}

Throughout the paper, we shall denote the argument of $\vagueweight_M$ in \eqref{eq:conjectural-mysterious-formula} by
\begin{equation}
\label{eq:def-gamma}
\gamma=\gamma(\alpha):=\frac{1-\alpha}{\sqrt{\alpha}}.   
\end{equation}

\begin{remark}
\label{remark:not-unique}
Let us explain a bit more the meaning of our conjecture.
We know that for each partition $\pi$ there exists a collection of polynomials $c_G(\gamma)$ indexed by
bicolored graphs such that
\begin{equation}
\label{eq:expansion-not-explicit}
\Ch_\pi^{(\alpha)} (\lambda) = 
\sum_G  c_G(\gamma)\ N^{(\alpha)}_G(\lambda)
\end{equation}
holds true for an arbitrary Young diagram $\lambda$ ---
see \cref{prop:ExistenceSumNG}.

Since the functions $N^{(\alpha)}_G$, seen as functions on the set of Young diagrams, 
are not linearly independent \cite[Proposition 2.2.1]{F'eray2009},
the expansion \eqref{eq:expansion-not-explicit} is not unique.
Also, as there might be several maps which correspond to a given graph, indexing the sum
by maps (instead of graphs) gives even more freedom on the coefficients.
Therefore our conjecture should be understood as a claim about the existence
of a \emph{particularly nice} expansion of $\Ch_\pi^{(\alpha)}$
in terms of the functions $N^{(\alpha)}_M$.
\end{remark}

\subsection{Our (unsuccessful) candidate for the weight $\vagueweight_M(\gamma)$}

As we have seen above, the case $\alpha=1$ corresponds to a summation over \emph{oriented maps} (Eq.~\eqref{eq:alpha=1}), 
while the cases $\alpha=2$ and $\alpha=\frac{1}{2}$ correspond to a summation over \emph{non-oriented maps}
(Eqs.~\eqref{eq:alpha=2}, \eqref{eq:alpha=1/2}) with some simple coefficients which depend only on some general 
features of the map, such as the number of the vertices.
Thus one can expect that the hypothetical weight $\vagueweight_M(\gamma)$ should be interpreted as a kind of 
\emph{measure of non-orientability of a given map $M$}. 

This notion of \emph{measure of non-orientability} is not very well defined.
For example, one could require that for $\alpha=1$ the corresponding coefficient $\vagueweight_M(0)$ 
is equal to $1$ if $M$ is orientable and zero otherwise; 
and that for $\alpha\in\left\{\frac{1}{2},2\right\}$ (which corresponds to $\gamma(\alpha)=\pm \frac{1}{\sqrt{2}}$) 
the coefficient $\vagueweight_M\left(\pm \frac{1}{\sqrt{2}}\right)$ takes some fixed value on all 
(orientable and non-orientable) maps.

In \cref{subsec:measure-of-nonorientability} we will define some quantity $\weight_M=\weight_M(\gamma)$
which attempts to measure the non-orientability of a given map $M$. 
Note that $\weight_M$ depends on $\gamma(\alpha)$ and thus 
implicitly on $\alpha$ as well but we drop this dependence in the notation.  
Roughly speaking, $\weight_M$ is defined as follows: 
we remove the edges of the map $M$ one after another in a random order. 
For each edge which is to be removed we check the \emph{type} of this edge (for example, an edge may be \emph{twisted} if, 
in some sense, it is a part of a \emph{M\"obius band}). 
We multiply the factors corresponding to the types of all edges. 
The quantity $\weight_M$ is defined as the mean value of this product.

One could complain that this is a weak measure of non-orientability of a map; 
in particular for $\alpha=1$ the corresponding weight $\weight_M$ does not vanish on non-orientable maps. 
Nevertheless, as we shall see, 
this weight $\weight_M$ gives a positive answer to our \cref{conj:main} in many (but, regretfully, not all!) cases. 

\subsection{Orientability generating series $\newCh_\pi^{(\alpha)} (\lambda)$}
We define the \emph{orientability generating series} as the formula
\eqref{eq:conjectural-mysterious-formula} from \cref{conj:main} in which the hypothetical
weight $\vagueweight_M$ has been substituted by the weight $\weight_M$ considered in the current paper:
\begin{equation}
\label{eq:conjectured-stanley}
\newCh_\pi^{(\alpha)} (\lambda):= 
(-1)^{\ell(\pi)} \ \sum_M \left(-1\right)^{|\V_\bullet(M)|}
 \weight_M(\gamma) \ N^{(\alpha)}_M(\lambda),  
\end{equation}
where the sum runs over all non-oriented maps $M$ with the face-type $\pi$.

Now it is natural to ask the following question.

\begin{question}
\label{conj:precise}
Does the weight $\weight_M$ give a positive answer to \cref{conj:main}?

In other words, is it true that for any partition $\pi$ and any Young diagram $\lambda$
the corresponding Jack character and the orientability generating series coincide:
\begin{equation}
\label{eq:our-dream-equality}
\Ch^{(\alpha)}_\pi(\lambda) = \newCh^{(\alpha)}_\pi(\lambda)?
\end{equation}
\end{question}

Regretfully, the answer to this question is \textbf{NO}:
with extensive computer calculations we were able to find some concrete counterexamples,
see \cref{sec:Conj_Is_False} for more details.
Nevertheless, as we shall discuss in the following, 
it seems that the formula \eqref{eq:conjectured-stanley} predicts \emph{some} properties 
of Jack characters surprisingly well.
For this reason we hope that the investigation of the quantity \eqref{eq:conjectured-stanley} 
might shed some light on the problem and eventually lead to the correct solution
of \cref{conj:main}.

For example, the positive answer for \cref{conj:precise} holds for the following special cases: 
$\pi=(n)$ which consists of a single part for $1\leq n\leq 8$, 
furthermore for $\pi=(2,2)$, and $\pi=(3,2)$ (the proofs are computer-assisted).
Furthermore, \cref{cor:add-parts-1} shows that the answer for \cref{conj:precise} 
is positive also for any of these partitions augmented by an arbitrary number of parts equal to $1$.

Also, the positive answer for \cref{conj:precise} holds in the special case $\alpha\in\left\{ 2, \frac{1}{2}\right\}$,
see \cref{thm:special-values-are-ok}.

However, in order to present the most interesting and promising type of predictions given by the formula
\eqref{eq:conjectured-stanley}, we will need the notion of \emph{Stanley polynomials}, see below.

\subsection{Stanley polynomials}
\label{Sect:Stanley_pol}
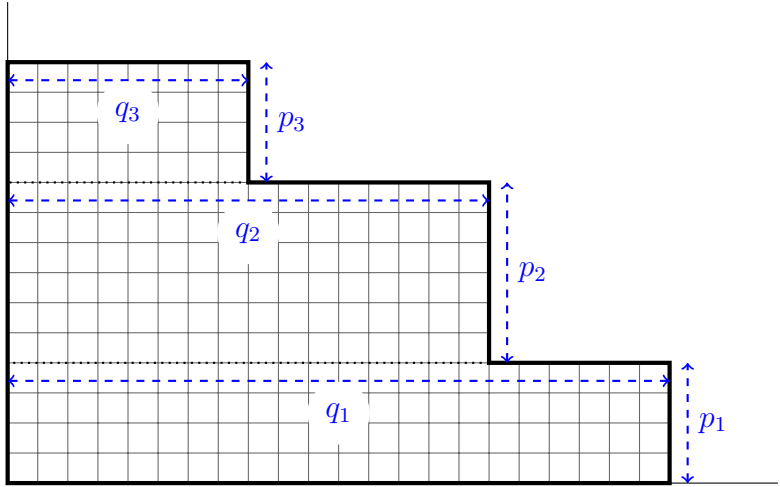
\begin{figure}[t]
\begin{tikzpicture}[scale=0.8]

\begin{scope}
\clip (0,0) -- (11,0) -- (11,2) -- (8,2) -- (8,5) -- (4,5) -- (4,7) -- (0,7) -- cycle;
\draw[black!60] (0,0) grid[step=0.5] (20,10); 
\end{scope}

\draw[ultra thick](0,0) -- (11,0) -- (11,2) -- (8,2) -- (8,5) -- (4,5) -- (4,7) -- (0,7) -- cycle;

\draw[dotted,thick] (8,2) -- (0,2)
(4,5) -- (0,5);

\begin{scope}[<->,thick,auto,dashed,blue]
\draw (11.3,0) to node[swap] {$p_1$} (11.3,2);
\draw (0,1.7) to node[swap,shape=circle,fill=white] {$q_1$} (11,1.7);

\draw (8.3,2) to node[swap] {$p_2$} (8.3,5);
\draw (0,4.7) to node[swap,shape=circle,fill=white] {$q_2$} (8,4.7);

\draw (4.3,5) to node[swap] {$p_3$} (4.3,7);
\draw (0,6.7) to node[swap,shape=circle,fill=white] {$q_3$} (4,6.7);
\end{scope}

\draw(0,0) -- (12.8,0);
\draw(0,0) -- (0,8);

\end{tikzpicture}

\caption{Multirectangular Young diagram $P\times Q$.}

\label{fig:multirectangular}
\end{figure}

If $P=(p_1,\dots,p_\ell)$ and $Q=(q_1,\dots,q_\ell)$ are sequences of non-negative integers 
such that $q_1\geq \cdots \geq q_\ell$, we consider the \emph{multirectangular} Young diagram
\[P \times Q := 
(\underbrace{q_1,\dots,q_1}_{\text{$p_1$ times}},\dots,\underbrace{q_\ell,\dots,q_\ell}_{\text{$p_\ell$ times}}).\]
This concept is illustrated in \cref{fig:multirectangular}. 
Quantities $P$ and $Q$ are referred to as \emph{multirectangular coordinates} of the Young diagram $P \times Q$.

Stanley \cite{Stanley2003/04} proved that the evaluation $\Ch^{(1)}_\pi(P\times Q)$ 
of the normalized character of the symmetric group on a multirectangular Young diagram
is a polynomial in the variables $p_1,\dots,p_\ell, q_1,\dots,q_\ell$.
This polynomial, referred to as \emph{Stanley polynomial}, turned out later to be a powerful tool
 in the context of \emph{Kerov polynomials}
\cite{DolegaF'eray'Sniady2008}.
In particular, one could argue that instead of viewing the character $\lambda\mapsto \Ch^{(1)}_\pi(\lambda)$
as a function of the Young diagram, it is more convenient to view it as a function 
\[(p_1,\dots,p_\ell,q_1,\dots,q_\ell)\mapsto  \Ch^{(1)}_\pi(P\times Q)\]
of the multirectangular coordinates. 
The idea of Stanley polynomials has to be a bit adjusted for the framework of Jack characters $\Ch^{(\alpha)}_{\pi}$;
we will discuss the details below.

If $c$ is a number and $P=(p_1,\dots,p_\ell)$, we use the shorthand notation 
\[ c P= c (p_1,\dots,p_\ell)=(c p_1,\dots, c p_\ell).\]
We will be interested in the evaluations of the Jack character 
\begin{equation}
\label{eq:stanley-for-alpha}
 \Ch^{(\alpha)}_{\pi} \left(  \sqrt{\alpha} P \times \frac{1}{\sqrt{\alpha}} Q\right)    
\end{equation}
as well as in analogous evaluations of the orientability generating series
\begin{equation}
\label{eq:stanley-for-alpha-ogs}
 \newCh^{(\alpha)}_{\pi} \left(  \sqrt{\alpha} P \times \frac{1}{\sqrt{\alpha}} Q\right)    
\end{equation}
in \emph{rescaled multirectangular coordinates} $p_1,\dots,p_\ell,q_1,\dots,q_\ell$.
Note that for the above expressions to make sense we have to assume that
\[\sqrt{\alpha}\ p_1,\dots,\sqrt{\alpha}\ p_\ell,\frac{1}{\sqrt{\alpha}} q_1,\dots,\frac{1}{\sqrt{\alpha}} q_\ell\]
are non-negative integers such that 
\[\frac{1}{\sqrt{\alpha}} q_1\geq \cdots \geq \frac{1}{\sqrt{\alpha}} q_\ell.\]
We will show (see \cref{lem:stanley-polynomials-exist}) that 
\eqref{eq:stanley-for-alpha} (as well as \eqref{eq:stanley-for-alpha-ogs})  
can be written uniquely as a polynomial in the indeterminates
\begin{equation}
\label{eq:stanley-coordinates-enchanced}
p_1,\dots,p_\ell,q_1,\dots,q_\ell,\gamma.   
\end{equation}
These polynomials will also be referred to as \emph{Stanley polynomials}.
It is worth pointing out that the indeterminate $\alpha$ is \emph{not} listed among the indeterminates
\eqref{eq:stanley-coordinates-enchanced}.

\subsection{Orientability Generating Series property}
\newcommand{\monomial}{\mathcal{M}}

As promised, we will discuss now the predictions given by the formula
\eqref{eq:conjectured-stanley} in the context of Stanley polynomials.

\subsubsection{The definition}
\begin{definition}[OGS-property]
\label[definition]{def:ogs-property}
Consider a pair which consists of 
\begin{itemize}
   \item a partition $\pi$, and
   \item a monomial
\begin{equation}
   \label{eq:generic-monomial}
 \monomial=p_1^{e_1} \cdots p_\ell^{e_\ell} \
   q_1^{f_1} \cdots q_\ell^{f_\ell} \
   \gamma^g
\end{equation}
in indeterminates \eqref{eq:stanley-coordinates-enchanced}.
\end{itemize}
We say that the pair $(\pi,\monomial)$ fulfills the \emph{Orientability Generating Series property (OGS-property)}
if the appropriate coefficients standing at the monomial $\monomial$ 
in the Stanley polynomials for Jack character \eqref{eq:stanley-for-alpha} and in the Stanley polynomial
for the orientability generating series \eqref{eq:stanley-for-alpha-ogs}
are equal:
\begin{equation} 
\label{eq:what-is-ogs}
[\monomial]\ \Ch^{(\alpha)}_\pi\left(  \sqrt{\alpha} P \times \frac{1}{\sqrt{\alpha}} Q\right) =
   [\monomial]\ \newCh^{(\alpha)}_\pi\left(  \sqrt{\alpha} P \times \frac{1}{\sqrt{\alpha}} Q\right). 
\end{equation}
\end{definition}

\bigskip

The question is: 
\emph{which pairs $(\pi,\monomial)$ have the OGS-property?}
There are known examples of pairs which \emph{do not} have this property 
(for example $\monomial=p_1 p_2 p_3 q_1 q_2 q_3$ and
$\pi=(9)$, see \cref{sec:Conj_Is_False}).
The positive examples will be discussed in the following.

\subsubsection{OGS-property and rectangular Young diagrams}
\label{subsec:rectangular}
Investigation of the normalized characters $\Ch_\pi(\lambda)$ in the case when $\lambda=p\times q$ is a 
\emph{rectangular} Young diagram was initiated by Stanley \cite{Stanley2003/04} who noticed 
that they have a particularly simple structure; 
in particular he showed that the formula \eqref{eq:alpha=1} holds true in this special case.

This line of research was continued by Lassalle \cite{Lassalle2008a} 
who (apart from other results) studied Jack characters $\Ch^{(\alpha)}_\pi(p\times q)$ on rectangular Young diagrams. 
In particular, Lassalle found a recurrence relation \cite[formula (6.2)]{Lassalle2008a} fulfilled by such characters; 
this recurrence relates the values of Jack characters on a fixed rectangular Young diagram $p\times q$, 
corresponding to various partitions $\pi$. This recurrence relation is essential for the current paper; 
expression \eqref{eq:conjectured-stanley} was formulated by a careful attempt to 
reverse-engineer the hypothetical hidden combinatorial structure behind Lassalle's recurrence.
In particular, our measure of non-orientability of maps $\weight_M$ was from the very beginning chosen in such a way 
that the formula \eqref{eq:our-dream-equality} from \cref{conj:precise}
holds true for an arbitrary rectangular Young diagram $\lambda=p \times q$.
In \cref{sec:rectangular} we will discuss these issues and prove the following theorem.
\begin{theorem}\
\label{thm:ogs-ok-for-rectangle}
\begin{itemize}
\item  
If $\lambda=p\times q$ is a rectangular Young diagram
then the answer for \cref{conj:precise} is positive; in other words
\[ \Ch^{(\alpha)}_\pi(p \times q) = \newCh^{(\alpha)}_\pi(p \times q)\]
holds true for any partition $\pi$ and any positive integers $p$ and $q$.

\item If a monomial $\monomial$ involves only variables $p_i,q_i,\gamma$ 
for some value of the index $i$
(i.e., $\monomial=p_i^e q_i^f \gamma^g$)
and $\pi$ is an arbitrary partition,
then $(\monomial,\pi)$ has the OGS-property, i.e.~the equality \eqref{eq:what-is-ogs} holds true.
\end{itemize}
\end{theorem}

Extensive computer exploration leads us to believe that the following extension of the above theorem holds true
as well.
\begin{conjecture}
\label{conj:2-rectangle}
\
\begin{itemize}
\item 
The answer for \cref{conj:precise} is positive if
$\lambda=(p_1,p_2) \times (q_1,q_2)$ is a multirectangular Young diagram consisting of (at most) two rectangles.

\item 
If monomial $\monomial$ 
involves only variables $p_i,p_j,q_i,q_j,\gamma$ for some choice of the indices $i,j$
(i.e., $\monomial=p_i^{e_i} p_j^{e_j} q_i^{f_i} q_j^{f_j} \gamma^g$)
and 
$\pi$ is an arbitrary partition,
then $(\monomial,\pi)$ has the OGS-property. 
\end{itemize}
\end{conjecture}

\subsubsection{OGS-property for top-degree monomials}
One can show that both Stanley polynomial for Jack character \eqref{eq:stanley-for-alpha}
and for the orientability generating series \eqref{eq:stanley-for-alpha-ogs}
are polynomials in indeterminates \eqref{eq:stanley-coordinates-enchanced} of degree 
$|\pi|+\ell(\pi)$. 
When this article was almost finished, the following positive result 
about the monomials of this maximal degree was announced.
\begin{theorem}[\cite{Sniady2014}]
Let $\pi$ be an arbitrary partition and let monomial $\monomial$ 
of the form \eqref{eq:generic-monomial} be of degree 
\[ e_1+\cdots+e_\ell+f_1+\cdots+f_\ell+g = |\pi|+\ell(\pi).\]

Then the pair $(\monomial, \pi)$ has the OGS-property.
\end{theorem}

This positive result is an indication that there could be indeed some truth behind 
our intuitions on the orientability generating series 
and it encourages further investigation of the topic.

\subsection{Links with other problems}
\subsubsection{Lassalle's conjectures}
\label{subsec:lassalles-conjectures}
Our search for some Stanley formula for general Jack characters $\Ch^{(\alpha)}_\pi$ 
has been motivated by two recent conjectures of Lassalle.
He conjectured some nice positivity and integrality properties
for the coefficients of $\Ch^{(\alpha)}_\pi$ expressed, respectively, 
in terms of \emph{multirectangular coordinates} (i.e., the coefficients of Stanley polynomials)
\cite[Conjecture 1 and Conjecture 2]{Lassalle2008a}
and in terms of \emph{free cumulants} (see \cref{sec:FreeCumulants} for the definition)
\cite[Conjecture 1.1 and Conjecture 1.2]{Lassalle2009}.
The positivity part in both conjectures would follow
from our 
Main Conjecture \ref{conj:main}:
for multirectangular coordinates we explain the connection
in \cref{sec:link_Lassalle}, while
for free cumulants this would follow from the general machinery
developed in \cite{DolegaF'eray'Sniady2008}.

\subsubsection{The $b$-Conjecture of Goulden and Jackson}

In the current paper we investigate the combinatorics of \emph{Jack characters} related to maps. 
The study of analogous connections between \emph{Jack polynomials} and maps is much older. 
In particular, Goulden and Jackson \cite{Goulden1996} formulated a conjecture (called \emph{$b$-Conjecture}) 
which claims, roughly speaking, that the connection coefficients of Jack polynomials can be
explained combinatorially as a summation over certain maps with coefficients that should describe 
\emph{non-orientability of a given map}. An extensive bibliography to this topic can be found in \cite{La2009}.

Although there is no direct link between our problem and the $b$-Conjecture 
(we are unable to show, for instance, that one
implies the other), both problems seem quite close and we hope that any progress on one of them 
could give ideas to solve the other.

\subsection{Outline of the paper}
In \cref{subsec:preliminaries}, we define Jack characters.
Then, in \cref{sec:maps}, we give formal definitions related to non-oriented
maps and define the weight $\weight_M$.
In \cref{sec:rectangular} and \cref{sec:special_alpha},
we prove that the answer for \cref{conj:precise} is positive
for rectangular Young diagrams on the one hand and,
on the other hand, for the special choice of the deformation parameter
$\alpha\in\left\{ \frac{1}{2}, 2\right\}$.
In \cref{sec:link_Lassalle}, we explain the link between 
\cref{conj:main} and the above-mentioned conjectures of Lassalle.
Then finally, in \cref{sec:Conj_Is_False},
we present our numerical exploration and the counterexample.

\section{Preliminaries}
\label{subsec:preliminaries}

\subsection{Partitions and Young diagrams}

A \emph{partition} $\pi=(\pi_1,\dots,\pi_l)$ is defined as a weakly decreasing finite sequence of positive integers. If $\pi_1+\cdots+\pi_l=n$ we also say that \emph{$\pi$ is a partition of $n$} and denote it by $\pi\vdash n$.
We will use the following notations: 
$|\pi|:=\pi_1+\cdots+\pi_l$;
furthermore $\ell(\pi):=l$ denotes the \emph{number of parts of $\pi$} and
\[m_i(\pi):=\big|\{k: \pi_k=i\}\big|\] 
denotes the \emph{multiplicity} of $i\geq 1$ in the partition $\pi$.
When dealing with partitions we will use the shorthand notation 
\[1^l:=(\underbrace{1,\dots,1}_{\text{$l$ times}}).\]

Any partition can be alternatively viewed as a \emph{Young diagram}. 
For drawing Young diagram we use the French convention.

The conjugacy classes of the symmetric group $\Sym{n}$ are in the one-to-one correspondence with the partitions of $n$. 
Thus any partition $\pi\vdash n$ can be also viewed as  some (arbitrarily chosen) permutation $\pi\in\Sym{n}$ 
with the corresponding cycle decomposition. 
This identification between partitions and permutations allows us to define characters 
such as $\Tr \rho(\pi)$, $\Ch^{(\alpha)}_{\pi}$ for $\pi$ being either a permutation or a partition.

\subsection{Jack characters}
\label{sec:JackCh}
\emph{Jack characters} were introduced by Lassalle \cite{Lassalle2008a,Lassalle2009}; 
however, we will use a slightly different normalization than the one used in his papers.
This new normalization was introduced and motivated in \cite{DolegaFeray2014}.
We present it in the following.

Firstly, as there are several of them, we have to fix a normalization of Jack polynomials. 
In our context it is most convenient to use the functions denoted by $J$ in the book of
Macdonald \cite[Section VI, Eq.~(10.22)]{Macdonald1995}. 
We expand the Jack polynomial in the basis of power-sum symmetric functions:
\[
J_\lambda^{(\alpha)}=\sum_{\substack{\rho: \\
|\rho|=|\lambda|}} 
\theta_{\rho}^{(\alpha)}(\lambda)\ p_{\rho};
\]
then we define the \emph{Jack character}
\begin{equation}
\label{eq:character}
\Ch_{\pi}^{(\alpha)}(\lambda):=
\alpha^{-\frac{|\pi|-\ell(\pi)}{2}}
\binom{|\lambda|-|\pi|+m_1(\pi)}{m_1(\pi)}
\ z_\pi \ \theta^{(\alpha)}_{\pi \cup 1^{|\lambda|-|\pi|}}(\lambda),
\end{equation}
where 
\[z_\pi = \pi_1 \pi_2 \cdots \ m_1(\pi)!\ m_2(\pi)! \cdots.\]
It turns out that for $\alpha=1$ we recover the usual
normalized characters \eqref{eq:definition-character} of the symmetric groups
(see \cite[Sections 1.3 and 1.4]{DolegaFeray2014}).

\subsection{Free cumulants}
\label{sec:FreeCumulants}
We will present now the notion of the \emph{free cumulants of a Young diagram}
which will be necessary later in some proofs.

\begin{definition}
    For a Young diagram $\lambda$, we define the sequence of its \emph{free cumulants} 
    $R^{(1)}_2(\lambda),R^{(1)}_3(\lambda),\dots$
    which are given (for $k \ge 2$) by
    \begin{equation}
    \label{eq:definition-free-cumulant}    
    R^{(1)}_k(\lambda) := (-1) \sum_{T} (-1)^{|\V_\bullet(T)|}\ N^{(1)}_{T}(\lambda),
    \end{equation}
    where the sum runs over \emph{bicolored, planted plane trees with $k$ vertices}.
    The formal definition of a planted plane tree will not be necessary for the purposes of this article; 
    we will only need the fact that each free cumulant
    is a linear combinations of the functions $N^{(1)}_{G}$.
    
\end{definition}

The usual way of defining the free cumulants of a Young diagram $\lambda$ is a two-step procedure \cite{Biane1998}: 
it uses \emph{Kerov's transition measure} of $\lambda$ 
\cite{Kerov1993transition} and the notion of \emph{free cumulants of a probability measure on the real line} \cite{NicaSpeicherBook}
which originates in Voiculescu's \emph{free probability theory}.
Our definition has the advantage of being more direct.
The equivalence between both definitions has been established by Rattan \cite{Rattan2007},
see also \cite[Theorem 9]{FeraySniady2011a}.

\begin{definition}
    Dealing with Jack polynomials, it is convenient to use \emph{anisotropic free cumulants} given by
    \begin{equation}
    \label{eq:definiton-anisotropic-free-cumulant}
    R^{(\alpha)}_k(\lambda) := (-1) \sum_{T} \left(-1\right)^{|\V_\bullet(T)|}
    N^{(\alpha)}_{T}(\lambda),
    \end{equation}
    where the sum runs again over \emph{bicolored, planted plane trees with $k\geq 2$ vertices}.
\end{definition}

The difference between \eqref{eq:definition-free-cumulant} and \eqref{eq:definiton-anisotropic-free-cumulant} 
lies in the use of the deformed number of embeddings $N^{(\alpha)}_T$ instead of the non-deformed one.
The anisotropic free cumulants have been introduced by Lassalle \cite{Lassalle2009}.
However, the normalization used here is different and corresponds to the one in \cite{DolegaFeray2014}
(for details see \cref{sec:lassalle's-normalization} below).

\subsection{Relationship to Lassalle's normalization}
\label{sec:lassalle's-normalization}
For Reader's convenience we provide below the relationship between the quantities used by Lassalle 
(in boldface) and the ones used by us:
\begin{align}
\notag    \bm{\vartheta}_{\pi \cup 1^{n-|\pi|}}^\lambda(\alpha) &=
    \binom{|\lambda|-|\pi|+m_1(\pi)}{m_1(\pi)}\ z_\pi\ \theta_{\pi \cup 1^{|\lambda|-|\pi|}}^\lambda(\alpha), 
\\
     \Ch^{(\alpha)}_\pi(\lambda) &=\alpha^{-\frac{|\pi|-\ell(\pi)}{2}}\ 
     \bm{\vartheta}_{\pi \cup 1^{n-|\pi|}}^\lambda(\alpha);
\label{eq:lassalle-normalization-character}
\\
\notag
R_k^{(\alpha)}(\lambda) &= \alpha^{k/2}\ \bm{R_k}(\lambda; \alpha).
\end{align}

Our convention has the advantage of being compatible with the symmetry 
$(\alpha,\lambda) \leftrightarrow (\alpha^{-1},\lambda')$,
where $\lambda'$ is the transpose diagram of $\lambda$ 
(defined, \emph{e.g.}, in \cite[Section 1.1]{Macdonald1995}).
Namely,
\begin{align}
\Ch_\pi^{(\alpha)}(\lambda) &= (-1)^{|\pi|-\ell(\pi)} \Ch_\pi^{(1/\alpha)}(\lambda');
\label{EqDuality}
\\
\notag
R_k^{(\alpha)}(\lambda) &= (-1)^k R_k^{(1/\alpha)}(\lambda').
\end{align}
The first equation is an adaptation of \cite[VI, (10.30)]{Macdonald1995}
to our notations,
while the second one follows easily from the definitions.

\subsection{Existence of an $N^{(\alpha)}_G$-expansion for Jack characters}

\begin{proposition}
\label{prop:existence-N-expansion}
For each partition $\pi$, there exists a collection $(c_G)$ of polynomials 
$c_G\in\QQ[\gamma]$ indexed by bicolored graphs such that
\begin{equation}
\label{eq:expansion-not-explicit-prop}
\Ch_\pi^{(\alpha)} (\lambda) = 
\sum_G  c_G\big(\gamma(\alpha) \big)\ N^{(\alpha)}_G(\lambda)
\end{equation}
holds true for an arbitrary Young diagram $\lambda$.
\label{prop:ExistenceSumNG}
\end{proposition}
\begin{proof}
    Any sum which is on the form given by the right-hand side of \eqref{eq:expansion-not-explicit-prop} 
    shall be called an \emph{$N^{(\alpha)}_G$-expansion}.

    By definition \eqref{eq:definiton-anisotropic-free-cumulant}, 
    any anisotropic free cumulant $R^{(\alpha)}_k$ admits an $N^{(\alpha)}_G$-expansion.
    Moreover, any linear combination and any product of some $N^{(\alpha)}_G$-expansions 
    is also an $N^{(\alpha)}_G$-expansion
    (indeed, it is straightforward to check that $N^{(\alpha)}_G \cdot N^{(\alpha)}_{G'}=N^{(\alpha)}_{G \sqcup G'}$).
    Also, the constant function equal to $\gamma$ has an $N^{(\alpha)}_G$-expansion corresponding 
    to the empty bicolored graph $\emptyset$:
    \[ \gamma = \gamma \ N^{(\alpha)}_{\emptyset}.\]
    Thus we have shown that each element of the algebra  
    generated by the following collection of functions on the set of Young diagrams:
    \[ \gamma, R^{(\alpha)}_2, R^{(\alpha)}_3,\dots \]
    has an $N^{(\alpha)}_G$-expansion.

    But, by \cite[Proposition 3.7]{DolegaFeray2014} the latter algebra contains each Jack character
    $\Ch_\pi^{(\alpha)}$, which concludes the proof.
 \end{proof}
Unfortunately, not much is known about the expansion of the Jack character $\Ch_\pi^{(\alpha)}$ 
in terms of the anisotropic free cumulants,
thus the coefficients $c_G$ in \eqref{eq:expansion-not-explicit} constructed 
by the above reasoning are not very explicit.
Moreover, let us recall that an $N^{(\alpha)}_G$-expansion --- if it exists --- 
is not unique (\cref{remark:not-unique})
and the expansion given by this proof is definitely \emph{not} the one we are looking for in \cref{conj:main}
(the above proof of \cref{prop:existence-N-expansion} gives a sum over \emph{forests}, 
while \cref{conj:main} requires that a coefficient of a forest should be independent of $\alpha$).

\subsection{Existence of Stanley polynomials}
\label{sec:existence-stanley-polynomials}

\begin{lemma}
\label{lem:stanley-polynomials-exist}
For each partition $\pi$, 
both the corresponding Jack character
\[ \Ch^{(\alpha)}_{\pi} \left(  \sqrt{\alpha} P \times \frac{1}{\sqrt{\alpha}} Q\right) \]
and the orientability generating series
\[ \newCh^{(\alpha)}_{\pi} \left(  \sqrt{\alpha} P \times \frac{1}{\sqrt{\alpha}} Q\right) \]
can be expressed (in a unique way) as a polynomial in the indeterminates \eqref{eq:stanley-coordinates-enchanced}.
\end{lemma}
\begin{proof}
We often viewed $\gamma=\gamma(\alpha)$ given by \eqref{eq:def-gamma} as a function of the parameter $\alpha$.
It will be convenient now to reverse the optics; for a given $\gamma\in\R$ we define $\alpha(\gamma)$ as
the unique positive solution of the equation
\[ \gamma = \frac{1-\alpha}{\sqrt{\alpha}}. \]

Let us fix the value of an integer $\ell\geq 1$.
We say that a function $F$ on the set of Young diagrams has \emph{the polynomiality property} if the map
\[ (p_1,\dots,p_\ell,q_1,\dots,q_\ell,\gamma) \mapsto 
F\left(  \sqrt{\alpha(\gamma)}\ P \times \frac{1}{\sqrt{\alpha(\gamma)}} Q\right) \]
coincides with some polynomial in indeterminates \eqref{eq:stanley-coordinates-enchanced}
(with coefficients in $\QQ$)
on the set
\[    \widetilde{PQ} := 
\left\{(P,Q,\gamma) : 
\sqrt{\alpha(\gamma)}\ P \times \frac{1}{\sqrt{\alpha(\gamma)}} Q 
\text{ is a Young diagram} \right\}.
\]

For any bicolored graph $G$, we claim that the function $N_G^{(\alpha)}$ has 
the polynomiality property. Indeed,
a slight variation of \cite[Lemma 3.9]{FeraySniady2011}
implies that $N_G^{(1)}(P \times Q)$ can be expressed as a polynomial (with non-negative, integer coefficients)
in the indeterminates $p_1,\dots,p_\ell,q_1,\dots,q_\ell$
with degree $|\V_\bullet(G)|$ in the variables $P$
and degree $|\V_\circ(G)|$ in the variables $Q$.
Therefore $N_G^{(\alpha)}(P \times Q)$ (see \eqref{eq:ng-alpha}) is
a polynomial in the indeterminates
\[ \frac{1}{\sqrt{\alpha}}\ p_1,\dots,\frac{1}{\sqrt{\alpha}}\ p_\ell,
\sqrt{\alpha}\ q_1,\dots,\sqrt{\alpha}\ q_\ell. \]
Equivalently, $N_G^{(\alpha)}$ has the polynomiality property.

Clearly, any linear combination of the functions $N_G^{(\alpha)}$ with coefficients in $\QQ[\gamma]$, 
also has the polynomiality property.
In particular (by \cref{prop:existence-N-expansion} and definition \eqref{eq:conjectured-stanley} respectively), 
so do $\Ch^{(\alpha)}_{\pi}$ and $\newCh^{(\alpha)}_{\pi}$.

\medskip

In order to prove uniqueness, it is enough to show that whenever a polynomial $F$ in indeterminates \eqref{eq:stanley-coordinates-enchanced}
vanishes on the set $\widetilde{PQ}$,
then $F$ is identically equal to $0$.
Consider such a polynomial $F$ and set $d$ to be its degree.
Let us fix $\alpha_0 >0$ and denote $\gamma_0 = \frac{1-\alpha_0}{\sqrt{\alpha_0}}$.
We define 
\begin{align*}
    S &= \{ i/\sqrt{\alpha_0} : 1 \leq i \leq d+1\}, \\ 
    S_j &= \{ (j(d+1)+i) \cdot \sqrt{\alpha_0} : 1\leq i \leq d+1\} \quad 
\text{for } 0 \leq j \leq l-1.
\end{align*}
Strictly from the definition, we have the following inclusion:
\[ \underbrace{S\times\cdots\times S}_{\text{$l$ factors}} \times 
S_{l-1} \times \cdots \times S_{0} \times \{\gamma_0\}\subset \widetilde{PQ},\]
so that $F$ vanishes on this Cartesian product.
It is easy to prove (for example, by induction over $n$, the number of the variables)
that, if a polynomial of degree $d$
\[ f \in \R[x_1,\dots,x_n]\]
vanishes on a Cartesian product $T_1\times\cdots\times T_n \subset \R^n$,
with $|T_i| > d$ (for $1 \le i \le n$), then
$f$ is identically equal to $0$.
Therefore --- for any fixed value of $\gamma_0$ --- $F(P,Q,\gamma_0)$ is equal to $0$
as a polynomial in $P$ and $Q$.
As this holds for any choice of the real number $\gamma_0$,
we conclude that
$F(P,Q,\gamma)$ is equal to $0$ as a polynomial in $P$, $Q$ and $\gamma$,
as wanted.
\end{proof}

\section{The measure of non-orientability of maps}
\label{sec:maps}

\subsection{Pairings and polygons}
\label{SubsectPairingsPolygons}
A \emph{set-partition} of a set $S$ is a collection \linebreak
$\{I_1,\cdots,I_r\}$ of pairwise disjoint, non-empty subsets,
the union of which is equal to $S$.

A \emph{pairing} (or, alternatively, \emph{pair-partition}) of $S$ 
is a set-partition into pairs.
If $s$ is an element of $S$ and $P$ is a pairing of $S$,
the \emph{partner} of $s$ in $P$ is defined as the unique element $t \in S$
such that $\{s,t\}$ is a pair of $P$.

For instance, for any integer $n \ge 1$,
\[P=\big\{ \{1,2\},\{3,4\},\ldots,\{2n-1,2n\} \big\} \]
is a pairing of $[2n]$ (we use the standard notation $[n]:=\{1,\cdots,n\}$).
Note that the existence of a pairing of $S$ clearly implies
that $|S|$ is even.

Let us consider now two pairings $\BC,\WC$ of the same set $S$ consisting of $2n$ elements.
We consider the following bicolored, edge-labeled graph $\loops(\BC,\WC)$:
\begin{itemize}
    \item it has $n$ black vertices indexed by the pairs of $\BC$
        and has $n$ white vertices indexed by the pairs of $\WC$;
    \item its edges are labeled with the elements of $S$.
        The extremities of the edge labeled $i$ are the pair of $\BC$ containing $i$
        and the pair of $\WC$ containing $i$.
\end{itemize}

Note that each vertex has degree $2$ and each edge has one white and one black
extremity. Besides, if we erase the indices of the vertices, it is easy to
recover them from the labels of the edges (the index of a vertex is the set of
the two labels of the edges leaving this vertex). Thus, in the following we forget the indices
of the vertices and view $\loops(\BC,\WC)$ as an edge-labeled graph.

As every vertex has degree $2$, the graph $\loops(\BC,\WC)$ can be seen as a collection of
polygons. Moreover, because of the proper bicoloration of the vertices, each polygon
has an even length. Let $2\ell_1\geq 2\ell_2\geq\cdots$ be the ordered lengths of
these polygons. The partition $(\ell_1,\ell_2,\dots)$ is called the \emph{type} of
$\loops(\BC,\WC)$ or the \emph{type} of the couple $(\BC,\WC)$. 

Special role will be played by the polygons having exactly $2$ edges. 
Such a polygon will be referred to as \emph{bigon}.

\begin{example}
\label{example:1}
For partitions
\begin{align*}
    \BC &= \big\{ \{1,2\},\{3,4\},\{5,6\},\{7,8\},\{9,10\}, \{A,B\},\{C,D\} \big\}, \\
    \WC &= \big\{ \{2,3\},\{4,5\},\{6,7\},\{8,9\},\{10,1\}, \{B,C\},\{D,A\} \big\}
\end{align*}
the corresponding polygons $\loops(\BC,\WC)$ are shown in \cref{fig:polygons}.
\end{example}

\newcommand{\faceA}{red!50}
\newcommand{\faceB}{blue}
\newcommand{\faceAfill}{red!10}
\newcommand{\faceBfill}{blue!20}

\begin{figure}[t]
\begin{tikzpicture}[scale=0.5,
black/.style={circle,thick,draw=black,fill=white,inner sep=4pt},
white/.style={circle,draw=black,fill=black,inner sep=4pt},
connection/.style={draw=black,black,auto}
]
\small

\fill[pattern color=\faceAfill,pattern=north west lines]  (0*36:5) -- (1*36:5) -- (2*36:5) -- (3*36:5) -- (4*36:5) -- (5*36:5) -- (6*36:5) -- (7*36:5) -- (8*36:5) -- (9*36:5);

\draw (0*36:5)  node (b1)     [black] {};
\draw (1*36:5)  node (b2)     [white] {};
\draw (2*36:5)  node (b3)     [black] {};
\draw (3*36:5)  node (b4)     [white] {};
\draw (4*36:5)  node (b5)     [black] {};
\draw (5*36:5)  node (b6)     [white] {};
\draw (6*36:5)  node (b7)     [black] {};
\draw (7*36:5)  node (b8)     [white] {};
\draw (8*36:5)  node (b9)     [black] {};
\draw (9*36:5)  node (b10)    [white] {};

\draw (0.5*36:4.1)  node {$1$};
\draw (1.5*36:4.1)  node {$2$};
\draw (2.5*36:4.1)  node {$3$};
\draw (3.5*36:4.1)  node {$4$};
\draw (4.5*36:4.1)  node {$5$};
\draw (5.5*36:4.1)  node {$6$};
\draw (6.5*36:4.1)  node {$7$};
\draw (7.5*36:4.1)  node {$8$};
\draw (8.5*36:4.1)  node {$9$};
\draw (9.5*36:4.1)  node {$10$};

\draw[connection] (b1) to  (b2);
\draw[connection] (b2) to  (b3);
\draw[connection] (b3) to  (b4);
\draw[connection] (b4) to  (b5);
\draw[connection] (b5) to  (b6);
\draw[connection] (b6) to  (b7);
\draw[connection] (b7) to  (b8);
\draw[connection] (b8) to  (b9);
\draw[connection] (b9) to  (b10);
\draw[connection] (b10)to (b1);

\begin{scope}[shift={(10,0)},scale=0.5]
\fill[\faceBfill]  (0:5) -- (90:5) -- (180:5) -- (270:5);

\draw (90*0:5)  node (b1)     [black] {};
\draw (90*1:5)  node (b2)     [white] {};
\draw (90*2:5)  node (b3)     [black] {};
\draw (90*3:5)  node (b4)     [white] {};

\draw (90*0.5:2.2) node {$A$};
\draw (90*1.5:2.2) node {$B$};
\draw (90*2.5:2.2) node {$C$};
\draw (90*3.5:2.2) node {$D$};

\draw[connection] (b1) to  (b2);
\draw[connection] (b2) to  (b3);
\draw[connection] (b3) to  (b4);
\draw[connection] (b4) to  (b1);
\end{scope}
\end{tikzpicture}

\caption{Polygons obtained from the couple of pairings from \cref{example:1}.}
\label{fig:polygons}

\vspace{3ex}

\begin{tikzpicture}[scale=0.6,
black/.style={circle,thick,draw=black,fill=white,inner sep=4pt},
white/.style={circle,draw=black,fill=black,inner sep=4pt},
connection/.style={draw=black,black,auto}
]
\scriptsize

\begin{scope}
\clip (0,0) rectangle (10,10);

\fill[pattern color=\faceAfill,pattern=north west lines] (0,0) rectangle (10,10);
\fill[\faceBfill] (5,5) rectangle (8,8);

\draw (2,7)  node (b1)     [black] {};
\draw (12,3) node (b1prim) [black] {};
\draw (3,2)  node (w1) [white] {};
\draw (-7,8) node (w1prim) [white] {};

\draw (5,5)  node (AA)     [black] {};
\draw (8,5)  node (BA)     [white] {};
\draw (5,8)  node (AB)     [white] {};
\draw (8,8)  node (BB)     [black] {};

\draw[connection]         (w1)      to node [pos=0.6] {$4$} node [swap,pos=0.6] {$9$} (AA);
\draw[connection]         (AA)      to node {$5$} node [swap] {$D$} (AB);
\draw[connection]         (AB)      to node {$6$} node [swap] {$C$} (BB);
\draw[connection]         (BB)      to node {$7$} node [swap] {$B$} (BA);
\draw[connection]         (BA)      to node {$8$} node [swap] {$A$} (AA);

\draw[connection]         (w1)      to node [pos=0.5] {$10$} node [swap,pos=0.425] {$2$} (b1prim);
\draw[connection]         (w1prim)  to node [pos=0.87] {$2$} node [swap,pos=0.95] {$10$} (b1);
\draw[connection]         (w1)      to node [pos=0.625] {$1$}  node [swap,pos=0.5] {$3$} (b1);

\end{scope}

\draw[very thick,decoration={
    markings,
    mark=at position 0.666  with {\arrow{>}}},
    postaction={decorate}]  
(0,0) -- (10,0);

\draw[very thick,decoration={
    markings,
    mark=at position 0.666  with {\arrow{>}}},
    postaction={decorate}]  
(10,10) -- (0,10);

\draw[very thick,decoration={
    markings,
    mark=at position 0.666  with {\arrow{>>}}},
    postaction={decorate}]  
(10,0) -- (10,10);

\draw[very thick,decoration={
    markings,
    mark=at position 0.666  with {\arrow{>>}}},
    postaction={decorate}]  
(0,10) -- (0,0);

\end{tikzpicture}
\caption{Example of a \emph{non-oriented map} drawn on the projective plane.
The left side of the square should be glued with a twist to the right side, 
as well as bottom to top (also with a twist), as indicated by arrows. 
This map has been obtained by gluing the edge-sides of the polygon from \cref{fig:polygons} 
according to the pair-partition given by Eq.~\eqref{eq:example-pairing}.}
\label{fig:map-projective-plane}

\end{figure}

\begin{definition}
\label{def:even-and-odd}
Let $s_1$ and $s_2$ be two elements of $S$ that belong to
the same polygon of $\loops(P_1,P_2)$.
Fix an arbitrary orientation of this polygon.
Then, one can consider the number of elements of $S$
between $s_1$ in $s_2$ in the polygon.
We say that $s_1$ and $s_2$ are in an \emph{even (respectively, odd) position}
if this number is even (respectively, odd).
As each polygon has an even size, this definition does not depend
on the choice of the orientation.
For example, in \cref{example:1} the elements $4$ and $9$ are in an even position, 
while the elements $1$ and $3$ are in an odd position.
\end{definition}

\subsection{Non-oriented maps}

\label{subsec:non-oriented}

The central combinatorial object in this paper is the following one.
\begin{definition}
\label{def:set-theoretic-definitions}
    An \emph{(unoriented) map} 
    is a triplet $M=(\BC,\WC,\E)$ of pairings of the same set $S$.
    \label{DefUmaps}
\end{definition}

The terminology comes from the fact that it is possible
to represent such a triplet of pair-partitions as a bicolored graph embedded in a non-oriented 
(and possibly non-connected) surface.
Let us explain how this works.

First, we consider the union of the polygons $\loops(\BC,\WC)$
defined in \cref{SubsectPairingsPolygons}.
The edges of these polygons, that is the elements of the set $S$,
are called \emph{edge-sides}.

We consider the union of the interiors of these polygons as a (possibly disconnected) surface
with a boundary.
If we consider two edge-sides, we can \emph{glue} them:
that means that we identify  with each other their white extremities, 
their black extremities, and the edge-sides themselves.

For any pair in the pairing $\E$, we glue the two corresponding edge-sides.
In this way we obtain a (possibly disconnected, possibly non-orientable)
surface $\varSigma$ without boundary.
After the gluing, the edges of the polygons form a bicolored graph $G$ embedded in the surface.
For instance, with the pairings $\BC$ and $\WC$ from \cref{example:1} and
\begin{equation}
\label{eq:example-pairing}
 \E = \big\{ \{1,3\},\{2,10\},\{4,9\},\{5,D\},\{6,C\}, \{7,B\},\{8,A\} \big\},  
\end{equation}
we get the graph from \cref{fig:map-projective-plane} embedded in the projective plane.

In general, the graph $G$ has as many connected components as the surface $\varSigma$.
Besides, the connected components of $\varSigma \setminus G$
correspond to the interiors of the collection of polygons we are starting from
and, thus, are homeomorphic to open discs.
These connected components are called \emph{faces}. 

This makes the link with the more common definition of the maps:
usually, a (non-oriented, bicolored) map is defined as a (bicolored)
\emph{connected} graph $G$ embedded in a (non-oriented) surface $\varSigma$ in such a way that
each connected component of $\varSigma \setminus G$ is homeomorphic to an open disc.
It should be stressed that with our definition --- contrary to the traditional convention --- 
we do not require the map to be \emph{connected}.

\begin{definition}
\label{def:set-theoretic-definitions2}
    Let $M=(\BC,\WC,\E)$ be a map.
    \begin{itemize}
        \item The elements of $\BC$ (respectively, $\WC$) are called black (respectively, white) \emph{corners}.
        \item The elements of $\E$ are called \emph{edges}; 
            we use the notation $\E(M)$ for the set of edges of a map
            (that is the third element of the triplet defining the map).
        \item The polygons $\loops(\BC,\WC)$ corresponding to the couple of pairings $(\BC,\WC)$
            are called \emph{faces}; the set of faces will be denoted by $\F(M)$.
            The \emph{face-type} of the map is the type
            of the couple $(\BC,\WC)$, as defined in \cref{SubsectPairingsPolygons}.
        \item The polygons $\loops(\BC,\E)$ (respectively, $\loops(\WC,\E)$)
            of the couple of pairings $(\BC,\E)$ (respectively, $(\WC,\E)$)
            are called \emph{black vertices} (respectively, \emph{white vertices});
            their set is denoted by $\V_\bullet(M)$ (respectively, $\V_\circ(M)$).
        \item A \emph{leaf} is a vertex of $M$ of degree $1$, that is
            a bigon of $\loops(\BC,\E)$ or $\loops(\WC,\E)$.
            In other terms, a leaf is a pair of edge-sides which belongs to both $\E$
            and $\BC$ or which belongs to both $\E$ and $\WC$.
        \item The \emph{connected components} of the map $M$ correspond
            to the connected components of the graph $G$ constructed above.
            Equivalently, they are the equivalence classes of the transitive
            closure of the relation: $x \sim y$ if $x$ is the partner of $y$
            in $\E$, $\BC$ or $\WC$.
    \end{itemize}
\end{definition}

Note that our maps have labeled edge-sides and
each element of $S$ is used exactly once as a label.

The pairing $\BC$ (respectively, $\WC$) indicates which edge-sides share
the same corner around a black (respectively, white) vertex.
This explains the names of these pairings.

This encoding of (non-oriented) maps by triplets of pairings
is of course not new.
It can for instance be found in \cite{Goulden1996a};
the presentation in that paper is nevertheless a bit different
as the authors consider there \emph{connected monochromatic} maps.

\subsection{Ribbon graphs}
\label{subsec:ribbon-graphs}

\begin{figure}[t]
\centering
\begin{tikzpicture}[scale=1,
white/.style={circle,thick,draw=black,fill=white,inner sep=4pt},
black/.style={circle,draw=black,fill=black,inner sep=4pt},
]
\begin{scope}
\clip (0,0) rectangle (5,5);
\fill[blue!3] (0,0) rectangle (5,5);
\node (v1) at (1.3,3.5) [black] {};
\node (v2) at (3.7,1.5) [white] {};
\draw (v2) +(-5,-0) to node [auto,swap,pos=0.7] {\footnotesize $2$} node [auto,pos=0.8] {\footnotesize $4$}
        (v1) to node [auto,swap] {\footnotesize $1$} node [auto] {\footnotesize $5$}
        (v2);
\draw (v1) +(5,0) to node [auto,swap,pos=0.7] {\footnotesize $4$} node [auto,pos=0.8] {\footnotesize $2$}
        (v2);
\draw (v1) to node [auto] {\footnotesize $3$} node [auto,swap] {\footnotesize $6$} (1.3,5);
\draw (v2) to node [auto] {\footnotesize $3$} node [auto,swap] {\footnotesize $6$} (3.7,0);
\end{scope}
\draw[very thick,decoration={
    markings,
    mark=at position 0.666  with {\arrow{>}}},
    postaction={decorate}]  
(0,0) -- (5,0);

\draw[very thick,decoration={
    markings,
    mark=at position 0.666  with {\arrow{>}}},
    postaction={decorate}]  
(5,5) -- (0,5);

\draw[very thick,decoration={
    markings,
    mark=at position 0.666  with {\arrow{>>}}},
    postaction={decorate}]  
(0,0) -- (0,5);

\draw[very thick,decoration={
    markings,
    mark=at position 0.666  with {\arrow{>>}}},
    postaction={decorate}]  
(5,0) -- (5,5);

\end{tikzpicture}
\caption{
Example of a map drawn on the Klein bottle: 
the left-hand side of the square should be glued to the right-hand one (without a twist) 
and the top side should be glued to the bottom one (with a twist), as indicated by the arrows.}
\label{fig:exampleA}

\centering
\begin{tikzpicture}[scale=0.5,
white/.style={circle,thick,draw=black,fill=white,inner sep=7pt},
black/.style={circle,draw=black,fill=black,inner sep=7pt},
connection/.style={auto,double distance=3mm}
]
\clip (-3,-8) rectangle (13,5);
\draw[black!80,connection] (0,0) .. controls (0,10) and (10,-20) .. (10,0); 
\draw[connection,draw=white,ultra thick] (0,0) .. controls (0,-20) and (10,10) .. (10,0); 
\draw[black!80,connection] (0,0) .. controls (0,-20) and (10,10) .. (10,0); 
\fill[fill=white,draw=white] (0,-3mm) rectangle (10,3mm);

\draw[draw=white,ultra thick] plot[smooth] file {ribbonStraightA.txt};
\draw[black!80] plot[smooth] file {ribbonStraightA.txt};
\draw[draw=white,ultra thick] plot[smooth] file {ribbonStraightB.txt};
\draw[black!80] plot[smooth] file {ribbonStraightB.txt};

\draw (0,0) node[black] {};
\draw (10,0) node[white] {};
\draw (4.5,0.3) node[above] {$3$};
\draw (4.5,-0.3) node[below] {$6$};

\draw (0,-4) node[anchor=east] {$1$};
\draw (0.5,-4) node[anchor=west] {$5$};

\draw (10,-4) node[anchor=west] {$4$};
\draw (9.5,-4) node[anchor=east] {$2$};
\end{tikzpicture}
\caption{The map from \protect\cref{fig:exampleA} drawn as a \emph{ribbon graph}.}
\label{fig:exampleB}
\end{figure}

For the purposes of the current paper it is sometimes convenient to represent 
a map as a \emph{ribbon graph} (see \cref{fig:exampleB}) as follows:
each vertex is represented as a small disc, 
and each edge is represented by a thin ribbon connecting two discs 
in a way that a walk along the boundary of the ribbons corresponds 
to the walk along the boundary of the faces of a given map. 

\subsection{Summation over maps with a specified face-type}
Summations over maps with a specified face-type
(such as in \cref{eq:conjectured-stanley,eq:alpha=2,eq:alpha=1/2})
should be understood as follows:
we fix a couple of pairings $(\BC,\WC)$ of type $\pi$
and consider all pairings $\E$ of the same ground set;
we sum over the resulting collection of maps $(\BC,\WC,\E)$.
The \emph{set of maps with a specified face-type} should be understood in an analogous way.

\subsection{Three kinds of edges}
\label{subsec:anatomy}

\begin{figure}[tbp]
\centering
\begin{tikzpicture}[scale=0.7,
black/.style={circle,thick,draw=black,fill=white,inner sep=4pt},
white/.style={circle,draw=black,fill=black,inner sep=4pt},
faceAs/.style={\faceA, dashed,  line width=6pt},
faceBs/.style ={\faceB, line width=6pt},
connection/.style={draw=black!80,black!80,auto}
]
\scriptsize

\begin{scope}
\clip (0,0) rectangle (10,10);

\fill[pattern color=\faceAfill,pattern=north west lines] (0,0) rectangle (10,10);

\draw (2,7)  node (b1)     [black] {};
\draw (12,3) node (b1prim) [black] {};
\draw (3,2)  node (w1) [white] {};
\draw (-7,8) node (w1prim) [white] {};

\draw (5,5)  node (AA)     [black] {};
\draw (8,5)  node (BA)     [white] {};
\draw (5,8)  node (AB)     [white] {};
\draw (8,8)  node (BB)     [black] {};

\draw[faceAs,left to-left to] (w1)      to (AA);
\draw[faceAs,left to-left to] (AA)      to (AB);
\draw[faceAs,left to-left to] (AB)      to (BB);
\draw[faceAs,left to-left to] (BB)      to (BA);
\draw[faceAs,left to-left to] (BA)      to (AA);
\draw[faceAs, ->] (w1)      to (b1prim);
\draw[faceAs, ->] (w1prim)  to (b1);
\draw[faceAs, <-] (w1)      to (b1);

\begin{scope}
\clip (5,5) -- (8,5) -- (8,8) -- (5,8);
\fill[\faceBfill] (5,5) rectangle (8,8);
\draw (5,5)  node (AA)     [black] {};
\draw (8,5)  node (BA)     [white] {};
\draw (5,8)  node (AB)     [white] {};
\draw (8,8)  node (BB)     [black] {};

\draw[faceBs, left to-left to] (AA) -- (AB);
\draw[faceBs, left to-left to] (AB) -- (BB);
\draw[faceBs, left to-left to] (BB) -- (BA);
\draw[faceBs, left to-left to] (BA) -- (AA);
\end{scope}

\draw[connection,draw=white,double=black,ultra thick]         (w1)      to node [pos=0.6] {\textcolor{black}{$4$}} node [swap,pos=0.6] {\textcolor{black}{$9$}} (AA);
\draw[connection,draw=white,double=black,ultra thick]         (AA)      to node {\textcolor{black}{$5$}} node [swap] {\textcolor{black}{$D$}} (AB);
\draw[connection,draw=white,double=black,ultra thick]         (AB)      to node {\textcolor{black}{$6$}} node [swap] {\textcolor{black}{$C$}} (BB);
\draw[connection,draw=white,double=black,ultra thick]         (BB)      to node {\textcolor{black}{$7$}} node [swap] {\textcolor{black}{$B$}} (BA);
\draw[connection,draw=white,double=black,ultra thick]         (BA)      to node {\textcolor{black}{$8$}} node [swap] {\textcolor{black}{$A$}} (AA);

\draw[connection,draw=white,double=black,ultra thick,pos=0.4] (w1)      to node [pos=0.5]  {\textcolor{black}{$10$}} node [swap,pos=0.425] {\textcolor{black}{$2$}} (b1prim);
\draw[connection,draw=white,double=black,ultra thick,pos=0.9] (w1prim)  to node [pos=0.82]  {\textcolor{black}{$2$}} node [swap,pos=0.9] {\textcolor{black}{$10$}} (b1);
\draw[connection,draw=white,double=black,ultra thick]         (w1)      to node [pos=0.625] {\textcolor{black}{$1$}} node [swap,pos=0.5] {\textcolor{black}{$3$}} (b1);

\end{scope}

\draw[very thick,decoration={
    markings,
    mark=at position 0.666  with {\arrow{>}}},
    postaction={decorate}]  
(0,0) -- (10,0);

\draw[very thick,decoration={
    markings,
    mark=at position 0.666  with {\arrow{>}}},
    postaction={decorate}]  
(10,10) -- (0,10);

\draw[very thick,decoration={
    markings,
    mark=at position 0.666  with {\arrow{>>}}},
    postaction={decorate}]  
(10,0) -- (10,10);

\draw[very thick,decoration={
    markings,
    mark=at position 0.666  with {\arrow{>>}}},
    postaction={decorate}]  
(0,10) -- (0,0);

\end{tikzpicture}
\caption{
The non-oriented map from \protect\cref{fig:map-projective-plane}.
On the boundary of each face some arbitrary orientation was chosen, as indicated by the arrows.
The edge $\{4,9\}$ is an example of a \emph{straight edge}, 
the edge $\{1,3\}$ is an example of a \emph{twisted edge}, 
the edge $\{6,C\}$ is an example of an \emph{interface edge}.}
\label{fig:map-orientation}

\end{figure}
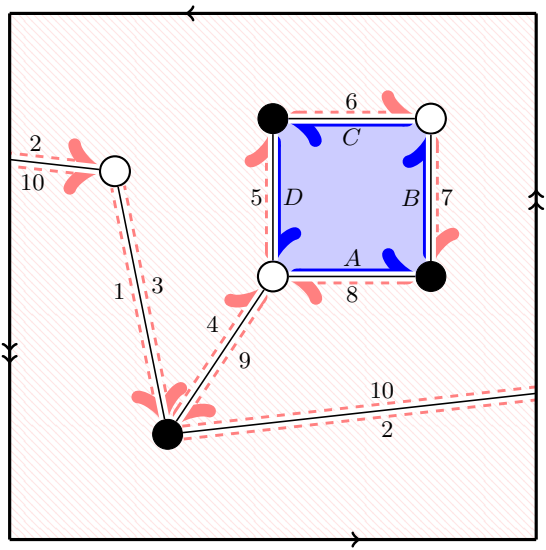

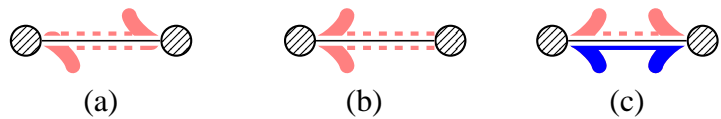
\begin{figure}[t]
\centering
\subfloat[]{
\begin{tikzpicture}[scale=0.5,
White/.style={circle,thick,draw=black,pattern=north east lines,inner sep=4pt},
Black/.style={circle,thick,draw=black,pattern=north east lines,inner sep=4pt},
connection/.style={draw=black!80,black!80,auto},
faceAs/.style={\faceA, dashed,  line width=6pt},
faceBs/.style ={\faceB, line width=6pt},
]

\draw[white] (-1,-1) rectangle (5,1);
\draw (0,0) node (b) [Black] {};
\draw (4,0) node (w) [White] {};

\draw[faceAs, line width=7pt, left to-left to] (w) to (b);
\draw[connection,draw=white,double=black,ultra thick] (w) to (b);
\end{tikzpicture}
\label{subfig:straight}
} 
\quad
\subfloat[]{
\begin{tikzpicture}[scale=0.5,
White/.style={circle,thick,draw=black,pattern=north east lines,inner sep=4pt},
Black/.style={circle,thick,draw=black,pattern=north east lines,inner sep=4pt},
connection/.style={draw=black!80,black!80,auto},
faceAs/.style={\faceA, dashed,  line width=6pt},
faceBs/.style ={\faceB, line width=6pt},
]

\draw[white] (-1,-1) rectangle (4,1);
\draw (0,0) node (b) [Black] {};
\draw (4,0) node (w) [White] {};

\draw[faceAs, line width=7pt, ->] (w) to (b);
\draw[connection,draw=white,double=black,ultra thick] (w) to (b);
\end{tikzpicture}
\label{subfig:mobius}
}
\quad
\subfloat[]{
\begin{tikzpicture}[scale=0.5,
White/.style={circle,thick,draw=black,pattern=north east lines,inner sep=4pt},
Black/.style={circle,thick,draw=black,pattern=north east lines,inner sep=4pt},
connection/.style={draw=black!80,black!80,auto},
faceAs/.style={\faceA, dashed,  line width=6pt},
faceBs/.style ={\faceB, line width=6pt},
]

\draw[white] (-1,-1) rectangle (4,1);
\draw (0,0) node (b) [Black] {};
\draw (4,0) node (w) [White] {};

\draw[faceAs, line width=7pt,<->] (w) to (b);
\begin{scope}
   \clip (-1,-1) rectangle (5,0);
   \draw[faceBs, line width=7pt, <->] (w) to (b);
\end{scope}

\draw[connection,draw=white,double=black,ultra thick] (w) to (b);
\end{tikzpicture}
\label{subfig:interface}
}
\caption{Three possible kinds of edges in a map (see \cref{fig:map-orientation}):
\protect\subref{subfig:straight} \emph{straight edge}: 
both edge-sides of the edge belong to the same face and have opposite orientations, 
\protect\subref{subfig:mobius} \emph{twisted edge}: both edge-sides of the edge belong to the same face 
and have the same orientation, 
\protect\subref{subfig:interface} \emph{interface edge}: the edge-sides of the edge belong to two different faces; 
their orientations are not important.
In all three cases the colors of the vertices are not important.}
\label{fig:3cases}
\end{figure}

Let a map $M$ with some selected edge $E=\{s_1,s_2\}$ be given.
We distinguish three cases (a schematic description and an example
of each case are given in Figures~\ref{fig:map-orientation} and
\ref{fig:3cases}): 
\begin{itemize}
 \item Both edge-sides $s_1$ and $s_2$ belong to the same face $F$
     \emph{and} are in an even position (see \cref{def:even-and-odd}).

     Graphically, this means that if we travel
     along the boundary of the face $F$ then we visit the edge $E$ twice \emph{and} 
the directions in which we travel twice
along the edge $E$ are \emph{opposite}, see \cref{subfig:straight}. 

In this case the edge $E$ is called \emph{straight} and we associate to it the
weight 
\[\weight_{M,E}:=1.\]

 \item Both edge-sides $s_1$ and $s_2$ belong to the same face $F$ 
     \emph{and} are in an odd position.
          
       Graphically, this means that if we travel
along the boundary of the face $F$, we visit the edge $E$ twice
\emph{and} the directions in which we travel twice
along the edge $E$ are \emph{the same}, see \cref{subfig:mobius}.

In this case the edge $E$ is called \emph{twisted} and we associate to it the
weight 
\[\weight_{M,E}:=\gamma= \frac{1-\alpha}{\sqrt{\alpha}}.\]

 \item Edge-sides $s_1$ and $s_2$ belong to different faces of the map,
     see \cref{subfig:interface}. 

In this case the edge $E$ is called \emph{interface} and we associate to it the
weight 
\[\weight_{M,E}:=\frac{1}{2}.\] 
\end{itemize}

\subsection{Removal of edges}

Let $P$ be a pairing of a set $S$ and $s_1, s_2$ be two distinct elements of $S$.
We define a pairing $P_{\{s_1,s_2\}}$
of the set $S \setminus \{s_1,s_2\}$ as follows:
\begin{itemize}
    \item if $\{s_1,s_2\}$ is a pair of $P$, 
        then \[P_{\{s_1,s_2\}}:=P \setminus \big\{\{s_1,s_2\}\big\};\]
    \item otherwise, consider the partners $t_1$ and $t_2$ of, respectively, $s_1$ and $s_2$.
        The elements $s_1,s_2,t_1,t_2$ are all distinct.
        We define
        \[P_{\{s_1,s_2\}}:=\left( P \setminus \big\{\{s_1,t_1\},\{s_2,t_2\}\big\} \right) \cup \{t_1,t_2\}. \]
        In other words, we remove the two pairs containing $s_1$ or $s_2$
        and we match together the two unmatched elements of $S \setminus \{s_1,s_2\}$.
\end{itemize}

Let $M=(\BC,\WC,\E)$ be a map and $E\in\E$ be an edge of $M$.
Then we define $M\setminus\{E\}$ (or, shortly, $M \setminus E$)
as the triplet 
\[
M\setminus\{E\}=M\setminus E:= (\BC_E,\WC_E,\E_E) 
\]
which consists of the three original pairings with the edge $E$ removed.

Graphically, this corresponds to the removal of the edge $E$ from the map $M$.
If one extremity (or both extremities) of this edge is a leaf (are leaves),
we also remove this vertex (both vertices). 
Note that a 
removal of an edge might drastically change
the topology of the surface on which the map is drawn;
for this reason it might be more convenient to use the graphical representation of a map as a \emph{ribbon graph}
(see \cref{subsec:ribbon-graphs}) 
because the effect of an edge removal from a map in this framework is much simpler, 
namely we just remove the appropriate edge from the corresponding ribbon graph.

\begin{lemma}
\label{lem:removal-commutative}
    Let $P$ be a pairing of a set $S$ and $s_1, s_2,s_3,s_4$ be four
    distinct elements of $S$.
    Then \[ \left(P_{\{s_1,s_2\}} \right)_{\{s_3,s_4\}}=
    \left(P_{\{s_3,s_4\}} \right)_{\{s_1,s_2\}}.\]
    \label{LemCommutativityRemoval}
\end{lemma}
\begin{proof}
    Easy case by case analysis.
\end{proof}

The above lemma implies that for any two distinct edges $E_1$, $E_2$ of a map $M$ one has
\[\big(M\setminus E_1 \big)\setminus E_2 =\big(M\setminus E_2 \big)\setminus E_1 ,\]
which allows to define the map 
\[ M \setminus \{E_1,\dots,E_i\}:=\Big( \big(M\setminus E_1 \big)\setminus \cdots \Big) \setminus E_i\] 
for an arbitrary subset $\{E_1,\dots,E_i\}\subseteq \E(M)$ of the edges of $M$ by an iterative procedure.
The fact that the above object is well-defined (i.e., it does not depend on the choice of the order
of the edges $E_1,\dots,E_i$)
can be also easily justified in a more intuitive way by referring to the ribbon graphs.

\subsection{Effect of an edge removal on the faces}
\label{subsec:effect-of-edge-removal-on-faces}

Let $E=\{s_1,s_2\}$ be one of the edges of the map $M$.
In the following we will investigate the faces of the map $M \setminus E$,
i.e.~the polygons associated to $(\BC_E,\WC_E)$.

\subsubsection{Removal of a straight edge}
\label{subsec:split}

Suppose that $E$ is a straight edge of the map $M$.
By definition, it means that the two edge-sides $s_1$ and $s_2$ of $E$
belong to the same face $F$ of $M$.
Besides, if we fix some orientation on the face $F$, there is
an even number, let us say $2i$, of the edge-sides between $s_1$ and $s_2$ in $F$.
With the other orientation there would also be an even number,
let us say $2j$, of the edge-sides between $s_1$ and $s_2$ in $F$.
In particular, the face $F$ consists of $2i+2j+2$ edge-sides. 
When we remove the edge $E$, the face $F$ is split into two faces $F_1$
and $F_2$ which consist of, respectively, $2i$ and $2j$ edge-sides
(in the degenerate case when $i$ or $j$ is equal to $0$,
the corresponding face does not exist).

This can be easily seen graphically, see \cref{fig:split}.
This figure, and the similar ones hereafter, should be understood as follows.
\begin{itemize}
    \item We represent on the figures only the polygon(s) containing the edge-sides
        $s_1$ and $s_2$; on the left-hand side in the map $M$ and on the right-hand side
        in the map $M \setminus E$. The other polygons are not modified by the edge removal.
    \item The edge-sides which are to be removed, namely $s_1$ and $s_2$, are indicated by thick dashed lines.
    \item For improved readability, the edge-side labels are written on the outside of the face.
        For $i\in\{1,2\}$ we denote by $t^\BC_i$ the partner of $s_i$ in $\BC$ and
	  by $t^\WC_i$ the partner of $s_i$ in $\WC$. 
           
  \item Consider the black (respectively, white) extremity of the edge $E=\{s_1,s_2\}$ in the map $M$.
      This vertex of $M$
       corresponds to the two black (respectively, white) vertices of the polygons,  
       namely to the black (respectively, white) extremities of the edge-sides $s_1$ and $s_2$;
       note that in a degenerate case these two extremities might coincide.  
       These two black (respectively, white) vertices are 
       decorated on the left-hand side of our figures with a square shape (respectively, with a diamond shape).
       After removal of the edge $E$ these two vertices are glued together into a single black (respectively, white)
       vertex which on the right-hand side of our figures is also decorated with the same square shape 
       (respectively, diamond shape).
\end{itemize}

\newcommand{\faceCfill}{black!50!green!10}
\begin{figure}[t]
    \begin{tikzpicture}
    \tikzstyle{bv}=[circle,fill=black,inner sep=0pt,minimum size=2mm]
    \tikzstyle{wv}=[circle,thick,draw=black,fill=white,inner sep=0pt,minimum size=2mm]
    \begin{scope}[xshift=.65cm]
    \coordinate (V1) at (67.5:2);
    \coordinate (V2) at (22.5:2);
    \coordinate (V3) at (-22.5:2);
    \coordinate (V4) at (-67.5:2);
    \end{scope}

    \begin{scope}[rotate=180,xshift=.65cm]
    \coordinate (W1) at (67.5:2);
    \coordinate (W2) at (22.5:2);
    \coordinate (W3) at (-22.5:2);
    \coordinate (W4) at (-67.5:2);
    \end{scope}

    \fill[pattern color=\faceAfill,pattern=north west lines]  
	  (W1) -- (W2) -- (W3) -- (W4) -- (V1) -- (V2) -- (V3) -- (V4) -- (W1);

    \node (v1) at (V1) [wv] {};
    \node (v2) at (V2) [bv] {};
    \node (v3) at (V3) [wv] {};
    \node (v4) at (V4) [bv] {};

    \draw (v1) to node [auto] {\footnotesize $t^{\BC}_2$} (v2);
    \draw[ultra thick,dashed] (v2) to node [auto] {\footnotesize $s_2$} (v3);
    \draw (v3) to node [auto] {\footnotesize $t^{\WC}_2$} (v4);

    \node (w1) at (W1) [bv] {};
    \node (w2) at (W2) [wv] {};
    \node (w3) at (W3) [bv] {};
    \node (w4) at (W4) [wv] {};

    \draw (w1) to node [auto] {\footnotesize $t^{\WC}_1$} (w2);
    \draw[ultra thick,dashed] (w2) to node [auto] {\footnotesize $s_1$} (w3);
    \draw (w3) to node [auto] {\footnotesize $t^{\BC}_1$} (w4);
        
    \node[draw=black, rectangle, minimum size=4mm] at (v2) {};
    \node[draw=black, rectangle, minimum size=4mm] at (w3) {};

    \node[draw=black, diamond, minimum size=4mm] at (v3) {};
    \node[draw=black, diamond, minimum size=4mm] at (w2) {};

    \draw[dashed] (w4) to (v1);
    \draw[dotted] (w1) to (v4);
    \draw[snake=brace] (w3)+(0,1.5) to node [auto] {$2j$ edges between $s_1$ and $s_2$} (2.49,2.27);
    \draw[snake=brace] (v3)+(0,-1.5) to node [auto] {$2i$ edges between $s_2$ and $s_1$} (-2.49,-2.27);
     \draw[->] (3.3,0) to node [auto] {removal of} 
     node [auto,swap] {a straight edge} (5.8,0);

     \begin{scope}[xshift=8cm,yshift=.76cm]
         \coordinate (NV2) at (0,0);
         \coordinate (NV1) at (30:1.5);
         \coordinate (NW4) at (150:1.5);

         \fill[pattern color=\faceBfill,pattern=horizontal lines]
             (NV2) -- (NW4) to [bend left] (NV1) -- (NV2);

         \draw node (nv2) at (NV2) [bv] {};
         \draw node (nv1) at (NV1) [wv] {};
         \draw node (nw4) at (NW4) [wv] {};

         \draw (nv1) to node [auto] {\footnotesize $t^{\BC}_2$} (nv2);
         \draw (nv2) to node [auto] {\footnotesize $t^{\BC}_1$} (nw4);
         \draw[dashed]  (nw4) to [bend left] (nv1);
         \draw[snake=brace] (-1.3,1.5) to node [auto] {$2j$ edges in total} (1.3,1.5);
     \end{scope}

     \begin{scope}[xshift=8cm,yshift=-.76cm]
         \coordinate (NW2) at (0,0);
         \coordinate (NW1) at (-30:1.5);
         \coordinate (NV4) at (-150:1.5);

         \fill[pattern color=\faceCfill,pattern=vertical lines]
             (NW2) -- (NV4) to [bend right] (NW1) -- (NW2);

         \draw node (nw2) at (NW2) [wv] {};
         \draw node (nw1) at (NW1) [bv] {};
         \draw node (nv4) at (NV4) [bv] {};
         \draw (nw2) to node [auto] {\footnotesize $t^{\WC}_2$} (nw1);
         \draw (nw2) to node [auto,swap] {\footnotesize $t^{\WC}_1$} (nv4);
         \draw[dotted]  (nv4) to [bend right] (nw1);

         \draw[snake=brace] (1.3,-1.5) to node [auto] {$2i$ edges in total} (-1.3,-1.5);

    \node[draw=black, diamond, minimum size=4mm] at (nw2) {};
    \node[draw=black, rectangle, minimum size=4mm] at (nv2) {};

     \end{scope}
    \end{tikzpicture}
    \caption{The impact on the faces of a removal of a straight edge.}
    \label{fig:split}
\end{figure}
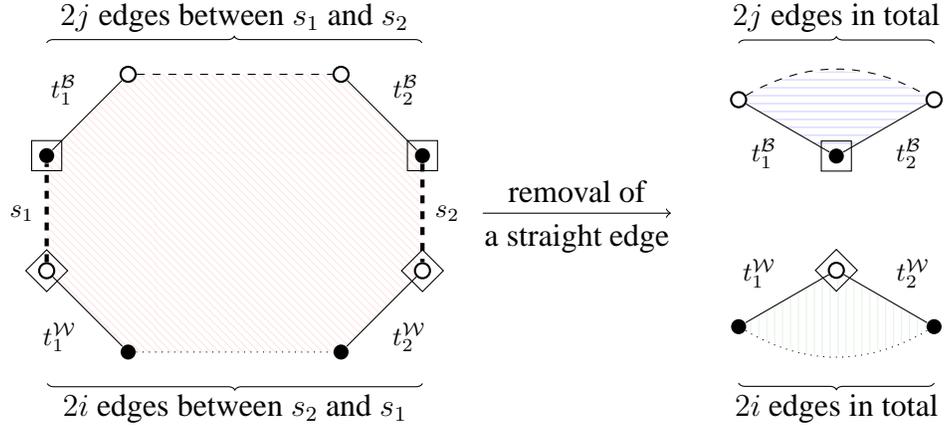

\subsubsection{Topological viewpoint on removal of an edge}
In order to give some topological intuition to the Reader,
we shall describe now and in the following 
(\cref{sec:topological-removal-straight-edge,sec:topological-removal-twisted-edge,sec:topological-removal-interface-edge})
the effect on the surface 
of a removal of an edge in each of the three cases.
As all the proofs in this paper are done rigorously with the 
\emph{``set theoretical point of view on maps''} 
(\cref{def:set-theoretic-definitions,def:set-theoretic-definitions2}, etc.),
we permit ourselves to stay a bit informal in these topological considerations.
For backgrounds on non-orientable surfaces, we refer to 
\cite[Section 2.1]{La2009} and the references therein.

We denote by $\varSigma$ (respectively, $\varSigma_e$) the surface on which
the map $M$ (respectively, $M \setminus E$) is naturally embedded; 
see \cref{subsec:non-oriented}.
We also denote by $G_M\subset \varSigma$ (respectively, by $G_{M\setminus E}\subset \varSigma_e$)
the graph associated to the map $M$ (respectively, to the map $M\setminus E$), viewed as a subset of the surface
$\varSigma$ (respectively, $\varSigma_e$).
Furthermore, $G_M\setminus E\subset \varSigma$ denotes the graph $G_M$ with the edge $E$ removed.

We shall compare:
\begin{enumerate}[label=(S\arabic*)]
   \item \label{surface:A}
the connected components of $\varSigma\setminus G_M$, i.e.~the faces of the map $M$,  
   \item \label{surface:B}
the connected components of
$\varSigma\setminus (G_M \setminus E)$, 
   \item \label{surface:C}
the connected components of $\varSigma_e\setminus (G_{M\setminus E})$, 
i.e.~the faces of the map $M\setminus E$.  
\end{enumerate}
Note that, by definition, each connected component in \ref{surface:A} and \ref{surface:C}
is homeomorphic to an open disc, while \emph{a~priori} some connected components in \ref{surface:B}
might be \emph{not} homeomorphic to a disc.
If this is indeed the case, then the surface $\Sigma$ does not fulfill
the defining property of the surface $\Sigma_e$ and thus 
$\varSigma\neq\varSigma_e$.

\subsubsection{Removal of a straight edge --- topological viewpoint}
\label{sec:topological-removal-straight-edge}
We come back now from the general considerations to the special case when the edge $E$ 
is a straight edge of the map $M$.
The edge-sides $s_1$ and $s_2$ of $E$ lie on the border of the same face $F$
(i.e., on the border of same connected component considered in \ref{surface:A}),
which --- by definition --- is homeomorphic to an open disc.
The parity condition on the relative position of the edge-sides $s_1$ and $s_2$ in $F$ implies that the 
situation depicted in \cref{subfig:straight} occurs.
Thus if we erase the edge $E$ and consider the corresponding surface $\Sigma\setminus(G_M\setminus E)$,   
the face $F$ becomes homeomorphic to an annulus, \emph{i.e.} the set $\{z \in \C: 1<|z|<2\}$.
In other terms, one of the components in \ref{surface:B} is an annulus
(each other is homeomorphic to an open disc).
In order to get the surface $\varSigma_e$, we must replace this annulus by two disjoint open discs
(both objects have the same border) which will be the elements of \ref{surface:C}.
This is in accordance with the fact that the face $F$ of $M$
is split in $M \setminus E$ into a pair of faces.

In the other direction, in order to transform $\varSigma_e$ into $\varSigma$,
one has to take two disjoint open discs lying on the surface $\varSigma_e$ and replace them by an annulus,
which corresponds to adding a \emph{handle} to the surface.

\subsubsection{Removal of a twisted edge}
\label{subsec:twist}
Suppose $E$ is a twisted edge of the map $M$.
By definition, it means that the two edge-sides $s_1$, $s_2$ of $E$
belong to the same face $F$ of the map $M$ and are in an odd position.
Let us denote by $2r$ the number of the edge-sides of the face $F$.
Then, after removal of the edge $E$, the face $F$ is replaced
by a face with $2(r-1)$ edge-sides; see \cref{fig:twist}.
The other faces are not modified.

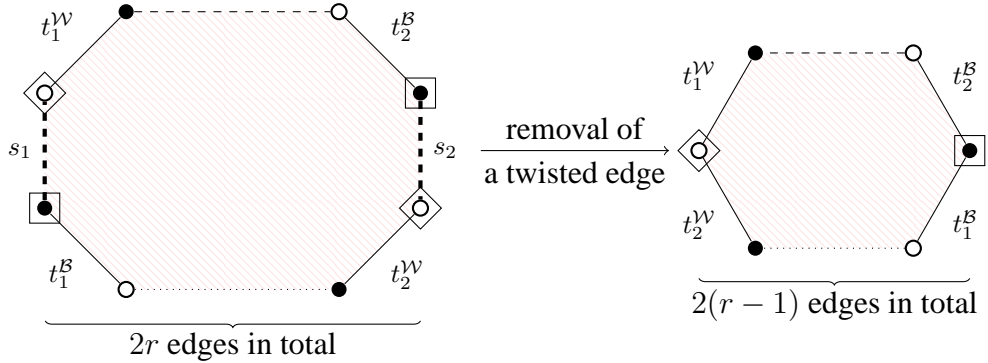
\begin{figure}[t]
    \begin{tikzpicture}
    \tikzstyle{bv}=[circle,fill=black,inner sep=0pt,minimum size=2mm]
    \tikzstyle{wv}=[circle,thick,draw=black,fill=white,inner sep=0pt,minimum size=2mm]
    \begin{scope}[xshift=.65cm]
    \coordinate (V1) at (67.5:2);
    \coordinate (V2) at (22.5:2);
    \coordinate (V3) at (-22.5:2);
    \coordinate (V4) at (-67.5:2);
    \end{scope}

    \begin{scope}[rotate=180,xshift=.65cm]
    \coordinate (W1) at (67.5:2);
    \coordinate (W2) at (22.5:2);
    \coordinate (W3) at (-22.5:2);
    \coordinate (W4) at (-67.5:2);
    \end{scope}

\fill[pattern color=\faceAfill,pattern=north west lines]  
(W1) -- (W2) -- (W3) -- (W4) -- (V1) -- (V2) -- (V3) -- (V4) -- (W1);

    \node (v1) at (V1) [wv] {};
    \node (v2) at (V2) [bv] {};
    \node (v3) at (V3) [wv] {};
    \node (v4) at (V4) [bv] {};

    \draw (v1) to node [auto] {\footnotesize $t^{\BC}_2$} (v2);
    \draw[ultra thick,dashed] (v2) to node [auto] {\footnotesize $s_2$} (v3);
    \draw (v3) to node [auto] {\footnotesize $t^{\WC}_2$} (v4);

    \node (w1) at (W1) [wv] {};
    \node (w2) at (W2) [bv] {};
    \node (w3) at (W3) [wv] {};
    \node (w4) at (W4) [bv] {};

    \draw (w1) to node [auto] {\footnotesize $t^{\BC}_1$} (w2);
    \draw[ultra thick,dashed] (w2) to node [auto] {\footnotesize $s_1$} (w3);
    \draw (w3) to node [auto] {\footnotesize $t^{\WC}_1$} (w4);        

    \draw[dashed] (w4) to (v1);
    \draw[dotted] (w1) to (v4);

    \node[draw=black, diamond, minimum size=4mm] at (w3) {};
    \node[draw=black, rectangle, minimum size=4mm] at (w2) {};
    \node[draw=black, diamond, minimum size=4mm] at (v3) {};
    \node[draw=black, rectangle, minimum size=4mm] at (v2) {};

    \draw[snake=brace] (v3)+(0,-1.5) to node [auto] {$2r$ edges in total} (-2.49,-2.27);
     \draw[->] (3.3,0) to node [auto] {removal of} 
     node [auto,swap] {a twisted edge} (5.8,0);
     
     \begin{scope}[xshift=8cm]
         \begin{scope}[xshift=.3cm]
             \coordinate (NV1) at (60:1.5);
             \coordinate (NV2) at (0:1.5);
             \coordinate (NV3) at (-60:1.5);
         \end{scope}
         \begin{scope}[rotate=180,xshift=.3cm]
             \coordinate (NW1) at (60:1.5);
             \coordinate (NW2) at (0:1.5);
             \coordinate (NW3) at (-60:1.5);
         \end{scope}
     \fill[pattern color=\faceAfill,pattern=north west lines]  
(NW1) -- (NW2) -- (NW3) -- (NV1) -- (NV2) -- (NV3) --  (NW1);

             \node (nv1) at (NV1) [wv] {};
             \node (nv2) at (NV2) [bv] {};
             \node (nv3) at (NV3) [wv] {};

             \node (nw1) at (NW1) [bv] {};
             \node (nw2) at (NW2) [wv] {};
             \node (nw3) at (NW3) [bv] {};

    \draw (nv1) to node [auto] {\footnotesize $t^{\BC}_2$} (nv2);
    \draw (nv2) to node [auto] {\footnotesize $t^{\BC}_1$} (nv3);
    \draw (nw1) to node [auto] {\footnotesize $t^{\WC}_2$} (nw2);
    \draw (nw2) to node [auto] {\footnotesize $t^{\WC}_1$} (nw3);
    \draw[dotted] (nv3) -- (nw1);
    \draw[dashed] (nv1) -- (nw3);
    \draw[snake=brace] (1.8,-1.7) to node [auto] {$2(r-1)$ edges in total}
        (-1.8,-1.7);
    \node[draw=black, diamond, minimum size=4mm] at (nw2) {};
    \node[draw=black, rectangle, minimum size=4mm] at (nv2) {};
     
	\end{scope}
     
    \end{tikzpicture}
    \caption{The impact on the faces of a removal of a twisted edge.}
    \label{fig:twist}
\end{figure}

\subsubsection{Removal of a twisted edge --- topological viewpoint}
\label{sec:topological-removal-twisted-edge}

Topologically, a removal of a twisted edge has the following effect.
Consider the map $M$ embedded in its surface $\varSigma$.
Again, both edge-sides $s_1$ and $s_2$ of $E$ are assumed to lie on the same face $F$ of $M$.
However, now the situation depicted on \cref{subfig:mobius} occurs and when we remove the edge $E$, 
the face $F$ becomes 
homeomorphic to a Möbius strip;
this Möbius strip,
denoted by $F'$,
is one of the connected components in \ref{surface:B}.
As the border of a Möbius strip is homeomorphic to a circle, 
one can remove $F'$ from $\varSigma$ and replace it by a disk;
this disk will be one of the elements of \ref{surface:C}.
In this way we obtain a surface $\Sigma'$ on which $G_M \setminus E$ is embedded in such a way that each
connected component of $\Sigma'\setminus(G_M\setminus E)$ is homeomorphic to an open disk.
This is the defining property of $\varSigma_e$, thus $\varSigma_e=\Sigma'$.

Conversely, $\varSigma$ is obtained from $\varSigma_e$ by replacing a disk by a Möbius strip or,
in other words, by adding a \emph{cross-cap}.

\subsubsection{Removal of an interface edge}
\label{subsec:join}
Suppose $E$ is an interface edge of $M$.
By definition, it means that the two edge-sides $s_1$, $s_2$ of $E$
belong to different faces $F_1$ and $F_2$ of $M$.
Let us denote by $2r$ and $2s$ the number of the edge-sides of the face $F_1$ and $F_2$, respectively.
Then after removal of the edge $E$, the faces $F_1$ and $F_2$ are replaced
by a single new face with $2(r+s-1)$ edge-sides; see \cref{fig:join}.
The other faces are not modified.

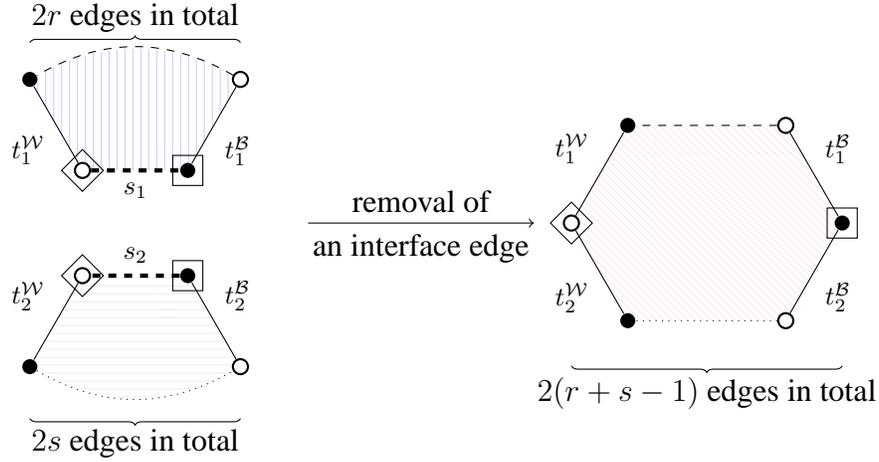
\begin{figure}[t]
    \begin{tikzpicture}
    \tikzstyle{bv}=[circle,fill=black,inner sep=0pt,minimum size=2mm]
    \tikzstyle{wv}=[circle,thick,draw=black,fill=white,inner sep=0pt,minimum size=2mm]
    \begin{scope}[yshift=0.7cm,scale=.7]
    \coordinate (V1) at (-2,1.73);
    \coordinate (V2) at (-1,0);
    \coordinate (V3) at (1,0);
    \coordinate (V4) at (2,1.73);
    \end{scope}  

    \fill[pattern color=\faceBfill,pattern=vertical lines]
             (V4) to [bend right] (V1) -- (V2) -- (V3);

    \node (v1) at (V1) [bv] {};
    \node (v2) at (V2) [wv] {};
    \node (v3) at (V3) [bv] {};
    \node (v4) at (V4) [wv] {};
    \draw (v1) to node [auto,swap] {\footnotesize $t^{\WC}_1$} (v2);
    \draw[ultra thick,dashed] (v2) to node [auto,swap] {\footnotesize $s_1$} (v3);
    \draw (v3) to node [auto,swap] {\footnotesize $t^{\BC}_1$} (v4);
    \draw[dashed] (v4) to [bend right] (v1);
        
    \node[draw=black, diamond, minimum size=4mm] at (v2) {};
    \node[draw=black, rectangle, minimum size=4mm] at (v3) {};
    \draw[snake=brace] (-1.4,2.4) to node [auto] {$2r$ edges in total} (1.4,2.4);
    \begin{scope}[rotate=180,yshift=.7cm,scale=.7] 
    \coordinate (W1) at (-2,1.73);
    \coordinate (W2) at (-1,0);
    \coordinate (W3) at (1,0);
    \coordinate (W4) at (2,1.73);
    \end{scope}  

    \fill[pattern color=\faceCfill,pattern=horizontal lines]
             (W4) to [bend right] (W1) -- (W2) -- (W3);

    \node (w1) at (W1) [wv] {};
    \node (w2) at (W2) [bv] {};
    \node (w3) at (W3) [wv] {};
    \node (w4) at (W4) [bv] {};
    \draw (w1) to node [auto,swap] {\footnotesize $t^{\BC}_2$} (w2);
    \draw[ultra thick,dashed] (w2) to node [auto,swap] {\footnotesize $s_2$} (w3);
    \draw (w3) to node [auto,swap] {\footnotesize $t^{\WC}_2$} (w4);
    \draw[dotted] (w4) to [bend right] (w1);

    \node[draw=black, diamond, minimum size=4mm] at (w3) {};
    \node[draw=black, rectangle, minimum size=4mm] at (w2) {};

    \draw[snake=brace] (1.4,-2.6) to node [auto] {$2s$ edges in total} (-1.4,-2.6);
     \draw[->] (2.3,0) to node [auto] {removal of} 
     node [auto,swap] {an interface edge} (5.3,0);
     \begin{scope}[xshift=7.6cm]
         \begin{scope}[xshift=.3cm]
             \coordinate (NV1) at (60:1.5);
             \coordinate (NV2) at (0:1.5);
             \coordinate (NV3) at (-60:1.5);
         \end{scope}
         \begin{scope}[rotate=180,xshift=.3cm]
             \coordinate (NW1) at (60:1.5);
             \coordinate (NW2) at (0:1.5);
             \coordinate (NW3) at (-60:1.5);
         \end{scope}

     \fill[pattern color=\faceAfill,pattern=north west lines]  
(NW1) -- (NW2) -- (NW3) -- (NV1) -- (NV2) -- (NV3) --  (NW1);

             \node (nv1) at (NV1) [wv] {};
             \node (nv2) at (NV2) [bv] {};
             \node (nv3) at (NV3) [wv] {};

             \node (nw1) at (NW1) [bv] {};
             \node (nw2) at (NW2) [wv] {};
             \node (nw3) at (NW3) [bv] {};
    \draw (nv1) to node [auto] {\footnotesize $t^{\BC}_1$} (nv2);
    \draw (nv2) to node [auto] {\footnotesize $t^{\BC}_2$} (nv3);
    \draw (nw1) to node [auto] {\footnotesize $t^{\WC}_2$} (nw2);
    \draw (nw2) to node [auto] {\footnotesize $t^{\WC}_1$} (nw3);
    \draw[dotted] (nv3) -- (nw1);
    \draw[dashed] (nv1) -- (nw3);
    \draw[snake=brace] (1.8,-1.9) to node [auto] {$2(r+s-1)$ edges in total}
        (-1.8,-1.9);
        
    \node[draw=black, diamond, minimum size=4mm] at (nw2) {};
    \node[draw=black, rectangle, minimum size=4mm] at (nv2) {};
     \end{scope}
    \end{tikzpicture}
    \caption{The impact on the faces of an interface edge removal.}
    \label{fig:join}
\end{figure}

\subsubsection{Removal of an interface edge --- topological viewpoint}
\label{sec:topological-removal-interface-edge}
Topologically, a removal of an interface edge is very simple.
Indeed, when we erase the edge $E$, the two faces $F_1$ and $F_2$ adjacent to $E$ in $M$ are merged into
a single connected component in \ref{surface:B} which is homeomorphic to an open disc.
So the condition that components of the surface without the graph
must be homeomorphic to open disks is not violated
and one can simply take $\varSigma_e=\varSigma$ as the surface on which 
$G_M \setminus E=G_{M\setminus E}$ is drawn.

\subsection{Weight associated to a map with a history}
Let $M$ be a map and let some linear order $\prec$ on its edges be given. This
linear order will be called
\emph{history}.

Let $E_1\prec \dots \prec E_n$ be the edges of $M$, listed according to the
linear order $\prec$. We set 
$M_i = M \setminus \{ E_1,\dots, E_i \}$ and define
\begin{equation}
\label{eq:weight-map-order}
 \weight_{M,\prec} := \prod_{0\leq i \leq n-1}  \weight_{M_{i}, E_{i+1}} 
 \end{equation}
(we recall that the quantity $\weight_{M,E}$ appearing on the right-hand side 
was defined in \cref{subsec:anatomy}).
This product $\weight_{M,\prec}$ can be interpreted as follows: from the map $M$ we remove (one by
one) all the edges, in the order
specified by the history. For each edge which is about to be removed we consider
its weight relative to the current map; finally we multiply all weights.

Please note that the type of a given edge (i.e.,~\emph{straight} versus \emph{twisted} versus \emph{interface}) 
might change in the process of removing edges and the weight $\weight_{M,\prec}$ 
usually depends on the choice of the history $\prec$.

\subsection{Measure of non-orientability of a map}
\label{subsec:measure-of-nonorientability}

Let $M$ be a map with $n$ edges. We define
\begin{equation}
\label{eq:take-average}
 \weight_M := \frac{1}{n!} \sum_{\prec} \weight_{M,\prec}. 
\end{equation}
This quantity can be interpreted as the mean value of the weight associated to
the map $M$ equipped with a randomly
selected history (with all histories having equal probability).
This is the central quantity for the current paper, we call it \emph{the measure of non-orientability} of the map $M$.

\begin{example}
\label{example}
We consider the map $M$ depicted in \cref{fig:exampleA}.
For calculations involving removal of edges it is more convenient to represent 
this map as a \emph{ribbon graph}, see \cref{fig:exampleB}.
For the history $\{3,6\} \prec \{1,5\} \prec \{2,4\}$ 
the corresponding weight is equal to $\weight_{M,\prec}=1 \cdot \frac{1}{2} \cdot 1$
while for the history $\{1,5\}\prec \{3,6\}\prec \{2,4\}$ the corresponding weight is equal to 
$\weight_{M,\prec}=\gamma \cdot \gamma \cdot 1$. The other histories are analogous to these two cases; finally
\[ \weight_M = \frac{\overbrace{1 \cdot \frac{1}{2} \cdot 1+1 \cdot \frac{1}{2} \cdot 1}^{\text{$2$ summands}}+
\overbrace{\gamma \cdot \gamma \cdot 1+\cdots+\gamma \cdot \gamma \cdot 1}^{\text{$4$ summands}} }{6}.\]
\end{example}

\begin{lemma}
\label{lem:how-twisted}
Let $M$ be a map. Then $\weight_M$ is a polynomial in variable $\gamma$ of degree (at most)
\[ \Moeb(M):=2 (\text{number of connected components of $M$}) - \chi(M), \]
where \[\chi(M):=|\V(M)|-|\E(M)|+|\F(M)|\]
is the Euler characteristic of $M$.

The polynomial $\weight_M(\gamma)$ is an even (respectively, odd) polynomial if and only if
the Euler characteristic $\chi(M)$ is an even number (respectively, an odd number).
\end{lemma}

Before the proof
notice that if $M$ is a 
\emph{connected} map 
then $\frac{1}{2}\Moeb(M)$ has a natural interpretation as the \emph{Euler 
genus} of the surface on which $M$ is drawn.
Moreover, if its underlying surface is orientable
then the map cannot have any twisted edges and
thus $\weight_M(\gamma)$ does not depend on $\gamma$.
Hence the degree of the polynomial $\weight_M(\gamma)$ satisfies
the same bound as some invariants defined by Brown and Jackson
\cite[Lemma 3.3]{BrownJacksonCandidateMatchingJackConjecture} and 
La Croix \cite[Theorem 4.4]{La2009}.

\begin{proof}[Proof of \cref{lem:how-twisted}]
We claim that for an arbitrary map $M$ and its edge $E$:
\begin{enumerate}[label=\emph{(\Alph*)}]
\item \label{littlelemma-A}
$\weight_{M,E}$ is a polynomial in $\gamma$ of degree (at most) $\Moeb(M)-\Moeb(M\setminus E)$,
\item \label{littlelemma-B}
$\weight_{M,E}$ is an even (respectively, odd) polynomial if and only if 
$\Moeb(M)-\Moeb(M\setminus E)$ is an even number (respectively, an odd number).
\end{enumerate}
Indeed, this statement follows by a careful investigation of 
each of the three cases considered in \cref{subsec:anatomy}
(cases when at least one of the endpoints of $E$
has degree $1$ must be considered separately).
We present the details in the following.
\begin{itemize}
    \item Assume that both extremities of $E$ are leaves.
        Then $E$ is clearly a straight edge, 
        so $\weight_{M,E}=1$.
        But $M \setminus E$ has two vertices less, one edge less, one face less
        and one connected component less than $M$. So $\Moeb(M \setminus E)=\Moeb(M)$
        and the claim holds in this case.
    \item Assume that exactly one extremity of $E$ is a leaf.
        Then $E$ is clearly a straight edge, 
        so $\weight_{M,E}=1$.
        But $M \setminus E$ has one vertex less, one edge less, and the same number of faces
        and connected components as $M$. So $\Moeb(M \setminus E)=\Moeb(M)$
        and claim holds in this case.
    \item Assume that $E$ is straight, but none of its extremities is a leaf.
        Recall that $\weight_{M,E}=1$ in this case.
        But $M \setminus  E$ has one edge less and one more face
        (see \cref{subsec:split}) than $M$.
        The number of vertices is unchanged, while the number of connected
        components can be constant or increase by $1$.
        In the first case, $\Moeb(M \setminus  E)=\Moeb(M)-2$ and in the second
        $\Moeb(M \setminus  E)=\Moeb(M)$.
        In both cases our claim holds.
    \item Assume that $E$ is twisted. In this case $\weight_{M,E}=\gamma$.
        Furthermore, maps $M$ and $M \setminus  E$ have the same number of
        connected components,
        the same number of faces and the same number of vertices (see \cref{subsec:twist}).
        However, $M\setminus E$ has one edge less than $M$.
        Thus $\Moeb(M \setminus  E)=\Moeb(M)-1$, and our claim holds in this case.
    \item Assume that $E$ is an interface edge. In this case $\weight_{M,E}=\frac{1}{2}$.
        Furthermore, maps $M$ and $M \setminus  E$ have the same number of
        connected components and the same number of vertices.
        However, the number of faces decreases by $1$ (see \cref{subsec:join}) as well as 
        the number of edges. 
        Thus $\Moeb(M \setminus  E)=\Moeb(M)$, and our claim holds in this case.
\end{itemize}

We apply the above claims \ref{littlelemma-A} and \ref{littlelemma-B} to each factor on the right-hand side of 
\eqref{eq:weight-map-order}; 
as the sum of the degrees of the factors is a telescopic sum,
it follows that the statement of the lemma holds true if the polynomial $\weight_M$ 
is replaced by $\weight_{M,\prec}$, where $\prec$ is an arbitrary history. 
The latter result finishes the proof by taking the average over $\prec$.
\end{proof}

\begin{remark}
    The above case-by-case analysis can be understood easily from
    the topological point of view, at least when the edge removal does not disconnect the graph.
    Indeed, a removal of a straight (respectively, twisted or interface) edge increases the Euler characteristic
    of the surface by $2$ (respectively, by $1$ or $0$), 
    see \cref{sec:topological-removal-straight-edge,sec:topological-removal-twisted-edge,sec:topological-removal-interface-edge}.
    It also increases the degree in $\gamma$ by $0$ (respectively, $1$ or $0$).
    In all cases, the degree in $\gamma$ increases at most as much as the Euler characteristic
    of the surface decreases, which explains our Lemma.
\end{remark}

\section{Rectangular Young diagrams and Lassalle's recurrence}
\label{sec:rectangular}

The main result of this section is the following one, which constitutes the first part of \cref{thm:ogs-ok-for-rectangle}
(the proof of the remaining part of \cref{thm:ogs-ok-for-rectangle} is postponed until \cref{sec:proof-OGS-rectangle}).

\begin{theorem}
\label{thm:rectangular-ok}
For any rectangular Young diagram 
\[\lambda = p \times q=(\underbrace{q,\dots,q}_{\text{$p$ times}}),\] 
and for an arbitrary partition $\pi$ and parameter $\alpha>0$,
the answer for \cref{conj:precise} is positive, i.e.,
\begin{equation}
\label{eq:equality-rectangular}
 \Ch_\pi^{(\alpha)}(p\times q) = \newCh_\pi^{(\alpha)} (p\times q)
 \end{equation}
holds true for an arbitrary partition $\pi$ and arbitrary positive integers $p$ and $q$.
\end{theorem} 

The main idea of the proof is to find a combinatorial interpretation of the
recurrence \eqref{eq:recurrence-lassalle} found by Lassalle \cite[formula (6.2)]{Lassalle2008a}.

\subsection{Lassalle's recurrence}

Following Lassalle \cite{Lassalle2008a} we denote by $\pi \cup (s)$ the partition $\pi$ with an extra part
$s$ added and by $\pi\setminus (s)$ the partition $\pi$ with one part $s$ removed 
(we will not use this notation when $\pi$ does not contain a part equal to $s$).
We also denote
\begin{align*}
 \pi \cup 1^l &= \pi \cup \underbrace{(1) \cup \cdots \cup (1)}_{\text{$l$ times}},
 \\ 
\pi_{\downarrow (s)} &= \big( \pi \setminus (s)\big) \cup (s-1), \\
 \pi_{\uparrow (rs)} &= \big( \pi \setminus (r+s+1)\big) \cup (r) \cup (s), \\
\pi_{\downarrow (rs)}
&= \Big( \big(\pi \setminus (r) \big) \setminus (s) \Big) \cup (r+s-1). 
\end{align*}

Consider a rectangular Young diagram $\lambda = p \times q$ and a partition $\pi$ such that $m_1(\pi) = 0$ 
(i.e., $\pi$ does not contain any part equal to $1$). Then Lassalle's recurrence relation
\cite[formula (6.2)]{Lassalle2008a}, after adapting to our normalizations takes the form:
\begin{multline}
\label{eq:recurrence-lassalle}
\left(\frac{p}{\sqrt{\alpha}}-\sqrt{\alpha} q\right) \sum_r r\ m_r(\pi)\
\Ch^{(\alpha)}_{\pi\downarrow(r)}(\lambda)\\
+\sum_r r\ m_r(\pi) \sum_{i=1}^{r-2}
\Ch^{(\alpha)}_{\pi\uparrow(i,r-i-1)}(\lambda)
\\
-\gamma \sum_r r(r-1)\ m_r(\pi)\
\Ch^{(\alpha)}_{\pi\downarrow(r)}(\lambda)\\
+ \sum_{r,s} rs\ m_r(\pi) \big( m_s(\pi)- \delta_{r,s} \big)\
\Ch^{(\alpha)}_{\pi\downarrow(rs)}(\lambda)\\
= - |\pi|\ \Ch^{(\alpha)}_\pi(\lambda). 
\end{multline}

The difficulty in this formula comes from the fact
that it was proved only under the assumption that $m_1(\pi)=0$.
In the following we will show how to overcome this issue.

\subsection{Partitions with parts equal to $1$}

Luckily, Jack characters corresponding a to partition $\pi$ with some parts equal to $1$
can be deduced from the case without parts equal to $1$.
Indeed, strictly from the definition of Jack characters \eqref{eq:character},
we have the following identity 
\begin{equation}
\label{eq:nielubie1A}
\Ch^{(\alpha)}_{\pi\cup 1^l}(\lambda) 
= \left(|\lambda|-|\pi|\right)_l\ \Ch^{(\alpha)}_\pi(\lambda).
\end{equation}

The following result shows that an analogous property is fulfilled by the orientability generating series $\newCh^{(\alpha)}_\pi$ as well.
\begin{lemma}
\label{lem:1part}
Let $\pi$ be a partition and $\lambda$ be an arbitrary Young diagram. Then
\begin{equation}
\label{eq:nielubie1B}
\newCh^{(\alpha)}_{\pi \cup 1^l}(\lambda) = (|\lambda| -
|\pi|)_l\
\newCh^{(\alpha)}_\pi(\lambda).
\end{equation}
\end{lemma}
\begin{proof}
It is enough to prove that for any partition $\pi$ the following holds:
\begin{equation}
 \label{eq:1part}
\newCh^{(\alpha)}_{\pi\cup(1)}(\lambda) = \left(|\lambda| -
|\pi|\right)\ \newCh^{(\alpha)}_\pi(\lambda).
\end{equation}

For a partition $\pi$ let $F_\pi$ be the set of pairs $(M', \prec)$,
where $M'$ is a map with the face-type $\pi$ and with a history $\prec$.
More explicitly: we start by fixing two pairings $\BC',\WC'$ of
the same set $S'$ so that $(\BC',\WC')$ has type $\pi$.
Then maps of face-type $\pi$ are triplets $(\BC',\WC',\E')$,
where $\E'$ is a pairing of $S'$.
In other words, each element of $F_\pi$ can be viewed as a pair-partition $\E'$
of $S'$, equipped with some linear order on the pairs.

Consider two elements $b_1$, $b_2$ that are \emph{not} in $S'$
and denote $S:=S' \sqcup \{b_1,b_2\}$.
We also consider the pairings $\BC:=\BC' \sqcup \{\{b_1,b_2\}\}$
and $\WC:=\WC' \sqcup \{\{b_1,b_2\}\}$ of $S$.
The couple $(\BC,\WC)$ has type $\pi \cup (1)$.
Hence maps of face type $\pi \cup (1)$ are triplets $M=(\BC,\WC,\E)$,
where $\E$ is an arbitrary pairing of $S$.

If $(M,\prec) \in F_{\pi\cup(1)}$ is such that 
$\{b_1,b_2\}$ is a pair of $\E$,
we will say that $(M,\prec) \in F^0_{\pi\cup(1)}$;
in the other case we say that $(M,\prec) \in F^1_{\pi\cup(1)}$.
This gives a disjoint decomposition 
\begin{equation}
\label{eq:decomposition}
F_{\pi\cup(1)} = F^0_{\pi\cup(1)} \sqcup F^1_{\pi\cup(1)}.
\end{equation}

\bigskip

Let $\HH\colon F_{\pi\cup(1)} \to F_\pi$ be a function 
$ \HH\colon (M,\prec) \mapsto (M',\prec')$ with $M'=(\BC,\WC,\E)$ defined as follows:
\begin{itemize}
\item
If $(M,\prec) \in F^0_{\pi\cup(1)}$,
then $\E':= \E \setminus \big\{ \{b_1,b_2\} \big\}$ 
is defined as the pairing $\E$
with the pair $\{b_1,b_2\}$ removed;
as the linear order $\prec'$
we take the restriction of $\prec$ to $\E'\subset \E$.

In this case $M$, viewed as a bicolored graph,
is a disjoint union of the bicolored graph $M'$
and the bicolored graph consisting of two vertices
connected by the edge $\{b_1,b_2\}$.
Thus the number of embeddings of $M$ into the Young diagram $\lambda$ 
is simply a product of $N^{(1)}_{M'}(\lambda)$ 
with the number of embeddings of a single edge into the Young diagram $\lambda$ 
(the latter number of embeddings is equal to $|\lambda|$, 
because any box of the Young diagram $\lambda$ can be the image of the edge, 
see \cref{sec:SubsectStanleyGeneral}).
The process of calculating $\weight_{M,\prec}$ is almost identical to the analogous process of calculating  
$\weight_{M',\prec'}$ except for the additional edge $\{b_1,b_2\}$ which is clearly a straight edge. Thus 
\begin{align*} N^{(1)}_M(\lambda) &=|\lambda|\ N^{(1)}_{M'}(\lambda), \\ 
|\V_\circ(M)| &= |\V_\circ(M')| + 1,\\ 
|\V_\bullet(M)| &= |\V_\bullet(M')| + 1, \\
\weight_{M,\prec} &= \weight_{M',\prec'}.\\
\intertext{The first three equations also imply that}
N^{(\alpha)}_M(\lambda) &= |\lambda|\ N^{(\alpha)}_{M'}(\lambda).  
\end{align*}

\item
Let $(M,\prec) \in F^1_{\pi\cup(1)}$.
The edge-sides $b_1$ and $b_2$ appear
in two different edges $\{e_1, b_i\},\{e_2, b_j\}\in \E$;
we choose the indices $\{i,j\}=\{1,2\}$ in such a way that $\{e_1, b_i\}\prec \{e_2, b_j\}$.
Then we set
\[\E' := \E_{\{b_1,b_2\}}= \left( \E \cup \big\{ \{e_1, e_2\} \big\}\right) \setminus 
\big\{  \{e_1, b_i\}, \{e_2, b_j\} \big\}.  \]
As the linear order $\prec'$ we take the unique linear order on $\E'$
that coincides with $\prec$ on the intersection $\E \cap \E'$ of their domains
and such that for any pair $P\in \E \cap \E'$ 
we have that \[P \prec' \{e_1, e_2\} \iff P \prec \{e_2, b_j\};\]
in other words the order $\prec'$ is obtained from $\prec$
by substituting the pair $\{e_2, b_j\}$ by $\{e_1, e_2\}$ and by removing the pair $\{e_1, b_i\}$.

The map $M$ is obtained from $M'$
by replacing the edge $\{e_1,e_2\}$
by a pair of edges in such a way that a new face is created 
(this face corresponds to the bigon $B=\{b_1,b_2\}$ in $\loops(\BC,\WC)$). 
It means that, as a bicolored graph, $M$ is obtained from the graph $M'$ by replacing one of the edges by a pair
of edges, hence $N^{(\alpha)}_M = N^{(\alpha)}_{M'}$.
The process of calculating $\weight_{M,\prec}$ is almost identical
to the analogous process of calculating  $\weight_{M',\prec'}$ except for
the edge $\{e_1,b_i\}$, which is the one of the two edges adjacent to the bigon $B$ which is removed first. 
This edge is clearly an interface edge. Thus
\begin{align*} 
N^{(1)}_M(\lambda) &= N^{(1)}_{M'}(\lambda), \\ 
 |\V_\circ(M)| &= |\V_\circ(M')|,\\ 
 |\V_\bullet(M)| &= |\V_\bullet(M')|, \\
  \weight_{M,\prec} &= \frac{1}{2} \weight_{M',\prec'}. \\
\intertext{The first three equations imply that}
N^{(\alpha)}_M(\lambda) &= N^{(\alpha)}_{M'}(\lambda).  
\end{align*}
\end{itemize}
This concludes the definition of the map $\HH$.

\bigskip

The left-hand side of \eqref{eq:1part} is equal to
\begin{multline}
\label{eq:decomposition-in-action}
\newCh^{(\alpha)}_{\pi\cup(1)}(\lambda) 
= \sum_{(M,\prec) \in F_{\pi\cup(1)}} \frac{(-1)^{\ell(\pi)+1}}{(|\pi|+1)!} 
\left(-1\right)^{|\V_\bullet(M)|}
  \weight_{M,\prec} \ N^{(\alpha)}_M(\lambda) \\
= \sum_{(M,\prec) \in F^0_{\pi\cup(1)}} \frac{(-1)^{\ell(\pi)}}{(|\pi|+1)!} 
\left(-1\right)^{|\V_\bullet(M')|}
  \weight_{M',\prec'} |\lambda| \ N^{(\alpha)}_{M'}(\lambda) \\
-\sum_{(M,\prec) \in F^1_{\pi\cup(1)}} \frac{(-1)^{\ell(\pi)}}{(|\pi|+1)!} 
\left(-1\right)^{|\V_\bullet(M')|}
  \weight_{M',\prec'} \frac{1}{2}\ N^{(\alpha)}_{M'}(\lambda),
\end{multline}
where $(M',\prec')$ is the image of $(M,\prec)$ by $\HH$
and in the last equality we used the decomposition \eqref{eq:decomposition}.

\bigskip

In the following we will show that for each $(M',\prec') \in F_\pi$ its preimage fulfills: 
\begin{align*} 
|\HH^{-1}(M', \prec') \cap F^0_{\pi \cup (1)}| &= |\pi|+1, \\
|\HH^{-1}(M', \prec') \cap F^1_{\pi \cup (1)}| &= 2\ (|\pi| + 1)|\ \pi | .
\end{align*}

First observe that for all pairs $(M,\prec)\in\HH^{-1}(M', \prec') \cap F^0_{\pi \cup (1)}$,
the maps $M$ are all the same: their edge pairing is given
by $\E=\E'\cup \big\{ \{b_1,b_2\} \big\}$.
Moreover, the order $\prec$ is obtained from the order $\prec'$
by adding the additional pair $\{b_1,b_2\}$ anywhere
between pairs of $\E'$ and this can be done in $|\pi|+1$ ways. 

Similarly, consider $(M,\prec)\in\HH^{-1}(M', \prec') \cap F^1_{\pi \cup (1)}$.
Then $\E$ is obtained from $\E'$ by removing some pair $\{e_1, e_2\}$
and adding the pairs $\{e_1, b_i\}$ and $\{e_2, b_j\}$
for some choice of $\{i,j\}=\{1,2\}$. 
Since the edges $e_1$ and $e_2$ play different roles (because by convention
$\{e_1, b_i\} \prec \{e_2, b_j\}$, 
we have altogether $4$ choices for doing this for each edge $\{e_1, e_2\}$.
The pair $\{e_1, e_2\}\in \E'$ can be equivalently specified by saying that there are $\ell$ elements 
which are smaller than $\{e_1, e_2\}$ (with respect to $\prec'$) with $0\leq \ell \leq |\pi|-1$.
The linear order $\prec$ is obtained by substituting the pair $\{e_1,e_2\}$ by $\{e_2,b_j\}$ and by adding the pair $\{e_1,b_i\}$ in such a way that $\{e_1,b_i\}\prec \{e_2,b_j\}$; there are $\ell+1$ choices for this. 
Thus the total number of choices is equal to
\[ \sum_{0\leq \ell \leq |\pi|-1} 4 (\ell+1) = 2\ (|\pi|+1)\ |\pi|,\]
just as we claimed.

\bigskip

Concluding, \eqref{eq:decomposition-in-action} gives us
\begin{multline*}
\newCh^{(\alpha)}_{\pi\cup(1)}(\lambda) \\
\shoveleft{= \sum_{(M', \prec') \in F_\pi} \frac{(-1)^{\ell(\pi)}}{(|\pi|)!\ (|\pi| + 1)} 
\left(-1\right)^{|\V_\bullet(M')|}}\\
\shoveright{\times \weight_{M',\prec'} \ N^{(\alpha)}_{M'}(\lambda) 
        \left( |\lambda|\ (|\pi|+1) - 2\ (|\pi| + 1)\ |\pi|\  \frac{1}{2} \right) 
}\\ 
 = (|\lambda| - |\pi|)
 \times \sum_{(M', \prec') \in F_\pi} \frac{(-1)^{\ell(\pi)}}{(|\pi|)!} 
\left(-1\right)^{|\V_\bullet(M')|}
  \weight_{M',\prec'} \ N^{(\alpha)}_{M'}(\lambda) \\
 = (|\lambda| - |\pi|)\ \newCh^{(\alpha)}_{\pi}(\lambda)
\end{multline*}
which finishes the proof.
\end{proof}

\begin{corollary}
\label{cor:add-parts-1}
If the answer for \cref{conj:precise} is positive for some partition $\pi$ and some Young diagram $\lambda$, 
it is also true for $\pi':=\pi\cup 1$ and $\lambda$. 
\end{corollary}
\begin{proof}
It is enough to use the recurrence relations \eqref{eq:nielubie1A} and \eqref{eq:nielubie1B}
\end{proof}

\subsection{Recurrence relation for the orientability generating series}
In this section we shall see that the orientability generating series $\newCh^{(\alpha)}_\pi(p \times q)$ 
--- when evaluated on a rectangular Young diagram --- fulfills a
recurrence relation analogous to \eqref{eq:recurrence-lassalle}.

There is an important simplification when we restrict our attention to a rectangular Young
diagram, thanks to the following lemma.
\begin{lemma}
For a rectangular Young diagram $\lambda=p\times q$,
the number of embeddings of a bicolored graph $G$ in $\lambda$
is given by the particularly simple formula
\[N^{(1)}_G(\lambda)=p^{|\V_\bullet(G)|}\ q^{|\V_\circ(G)|}.\]
    \label{lem:NGRectangular}
\end{lemma}
\begin{proof}
    It is a particular case of \cite[Lemma 3.9]{FeraySniady2011}.
\end{proof}
We can now prove the following.
\begin{proposition}
\label{prop:recurrence}
If $\lambda=p\times q$ is a rectangular Young diagram and
$\pi$ is a partition such that $m_1(\pi)=0$ then
\begin{multline}
  \label{eq:recurrence-lassalle-prim}
  \underbrace{\left(\frac{p}{\sqrt{\alpha}}-\sqrt{\alpha} q\right) \sum_{r \ge 1} r\ m_r(\pi)\
\newCh^{(\alpha)}_{\pi\downarrow(r)}(\lambda)}_\text{
(removing a leaf)}
 \\
 + \underbrace{\sum_{r \ge 1} r\ m_r(\pi) \sum_{i=1}^{r-2}
\newCh^{(\alpha)}_{\pi\uparrow(i,r-i-1)}(\lambda)}_\text{ (removing a
straight edge)}\\
- \underbrace{\gamma \sum_{r \ge 1} r(r-1)\ m_r(\pi)\
\newCh^{(\alpha)}_{\pi\downarrow(r)}(\lambda)}_\text{
(removing a twisted edge)} \\
+ \underbrace{\sum_{r,s \ge 1} rs\ m_r(\pi) \big( m_s(\pi)- \delta_{r,s} \big)\
\newCh^{(\alpha)}_{\pi\downarrow(rs)}(\lambda)}_\text{ (removing an interface edge)}\\
= - |\pi|\ \newCh^{(\alpha)}_\pi(\lambda). 
\end{multline}
 \end{proposition}
The comments under the curly braces concerning the individual summands on the left-hand side of 
this recurrence relation are connected to its proof; see below.

\begin{remark}
Notice that the assumption $m_1(\pi) = 0$ in the proposition above means that all the sums on the left hand side of \eqref{eq:recurrence-lassalle-prim} can, alternatively, run over $r \geq 2$ instead of $r \geq 1$.
\end{remark}

\begin{proof}
Clearly, \eqref{eq:weight-map-order} is equivalent to a recursive
relationship
\[ \weight_{M,\prec} = \weight_{M,E}\cdot \weight_{M\setminus E, \prec'}, \]
where $E$ is the first edge according to the linear order $\prec$ and $\prec'$ is a restriction of $\prec$ to the edges of $\E(M\setminus E)$.

Using \eqref{eq:conjectured-stanley}, \eqref{eq:ng-alpha}, 
\eqref{eq:take-average} and \cref{lem:NGRectangular}, 
the right-hand side of \eqref{eq:recurrence-lassalle-prim} can be written as
\begin{multline}
\label{eq:trzej-krolowie}
-|\pi|\ \newCh^{(\alpha)}_\pi(\lambda) \\
= \frac{(-1)^{\ell(\pi)-1}}{ (|\pi|-1)!} 
\sum_{(M,\prec)}
\left(- \frac{p}{\sqrt{\alpha}}\right)^{|\V_\bullet(M)|}
 \left( q\ {\sqrt{\alpha}}\right)^{|\V_\circ(M)|}\ \weight_{M,\prec},
\end{multline}
where the sum runs over maps of face type $\pi$ with a specific history.
In the following we will use the notation
\[\contribution_{M,\prec}(\lambda) := 
\left(- \frac{p}{\sqrt{\alpha}}\right)^{|\V_\bullet(M)|}
 \left(q\  {\sqrt{\alpha}}\right)^{|\V_\circ(M)|}\ \weight_{M,\prec} \]
for the contribution of the pair $(M,\prec)$ 
to the sum on the right-hand side of \eqref{eq:trzej-krolowie}.

Recall that the summation in \eqref{eq:trzej-krolowie} should be interpreted
as follows: we fix a couple $(\BC,\WC)$ of pairings of type $\pi$
and we sum over all choices of the pairing $\E$ of the same ground set $S$;
we also sum over all choices of the linear order $\prec$ on $\E$.
For such a map $M=(\BC,\WC,\E)$ and a linear order $\prec$,
we denote by $E=\{s_1,s_2\}$ the first edge
according to the linear order $\prec$
and by $\prec'$ the restriction of the
linear order $\prec$ to the edges of $\E \setminus \{E\}$.

The summation over $(M,\prec)$ can be seen alternatively as follows:
we first choose the first edge $E$ and then sum over all choices of the couple $(\E',\prec')$
where $\E'=\E \setminus \{E\}$ is a pairing of $S \setminus E$
and $\prec'$ a linear order on $\E'$.
This summation over $\E'$ can be interpreted as a summation over all choices of the map
$M'=(\BC_E,\WC_E,\E')$ of face-type corresponding to the 
type of the couple $(\BC_E,\WC_E)$.
Note that the map $M'$ corresponds to $M \setminus E$.
We shall use this idea repetitively in the proof.

We will split our sum depending on the type (straight, twisted or interface)
of the edge $E$ in the map $M$. 
Note that, as $\BC$ and $\WC$ are fixed, this type depends only on
the pair $E$, not on the remaining pairs in $\E$.
According to the classification from \cref{subsec:anatomy} there are the following possibilities:
\bigskip

\emph{The edge $E$ is straight and both endpoints of $E$ have degree $1$.}
This is not possible since it would imply that one of the faces of $M$ is a bigon thus $m_1(\pi)\geq 1$.
\bigskip

\emph{The edge $E$ is straight (hence $\weight_{M,E}=1$) and
 only the black (respectively, white) endpoint of $E$ has degree $1$
 (i.e., it is a leaf).}
    In other terms, the pair $E$ belongs also to the pairing $\BC$
    (respectively, $\WC$).
    We consider the map $M\setminus E$;
recall that it has one black (respectively,
white) vertex less than $M$ (the leaf extremity has been removed together with the edge).
It follows that 
\begin{align}	      \contribution_{M,\prec}(\lambda) &= \frac{-p}{\sqrt{\alpha}} 
\contribution_{M\setminus E, \prec'}(\lambda);
\label{eq:remove_black_leaf}\\
\intertext{respectively,} 
\notag
      \contribution_{M,\prec}(\lambda) &=  q\ \sqrt{\alpha}\  
\contribution_{M\setminus E, \prec'}(\lambda).
\end{align}

Fix the black vertex $B=\{b_1,b_2\} \in \BC$ and let us consider the total contribution
of the couples $(\E,\prec)$ such that $B$ is the black endpoint of $E$; in other words $E=B$.
This means that $\E = \E' \sqcup \{B\}$,
where $\E'$ is a pairing of $S \setminus B$, and thus $\BC_B=\BC\setminus \{B\}$. Therefore
$(\BC_E, \WC_E)=(\BC_B, \WC_B)$ is a couple of pairings of $S \setminus B$
of type $\pi_{\downarrow (r)}$, where $2r$ is the number of edge-sides
in the polygon of $\loops(\BC,\WC)$ containing $B$.
It follows that the summation over all choices of the pairing $\E'$ corresponds to 
a summation over all choices of the map $M'$ of face-type $\pi_{\downarrow (r)}$.
By definition, the map $M'=(\BC_B, \WC_B,\E')$ is equal to $M \setminus E$
and its contribution appears as the last factor on the right-hand side of Equation \eqref{eq:remove_black_leaf}
above.

Therefore, for a fixed pair $B \in \BC$ that belongs
to a polygon of size $2r$ of $\loops(\BC,\WC)$, the total contribution
to the right-hand side of \eqref{eq:trzej-krolowie}
of the couples $(\E,\prec)$ as above is given by
\[ \frac{(-1)^{\ell(\pi)-1}}{(|\pi|-1)!} \sum_{(M',\prec')} 
\frac{-p}{\sqrt{\alpha}}  \contribution_{M',\prec'}(\lambda)
= \frac{p}{\sqrt{\alpha}}\ \newCh_{\pi_{\downarrow (r)}}(\lambda). \]
For each $r\geq 2$ there are $m_r(\pi)$ polygons of size $2r$ in $\loops(\BC,\WC)$, and each of them
contains $r$ pairs of $\BC$, 
hence the total contribution of the pairs $(M,\prec)$ such that
the first edge $E$ belongs to $\BC$ (i.e., its black extremity is a leaf) is equal to 
\[ \frac{p}{\sqrt{\alpha}}\ \sum_{r \ge 2} r m_r(\pi)
\ \newCh_{\pi_{\downarrow (r)}}(\lambda).\]

Symmetrically, the total contribution of pairs $(M,\prec)$ such that
the first edge $E$ belongs to $\WC$ is equal to
\[-q\ \sqrt{\alpha}\ \sum_{r \ge 2} r m_r(\pi)
\ \newCh_{\pi_{\downarrow (r)}}(\lambda).\]

Finally, both cases together yield the first term of
the induction relation \eqref{eq:recurrence-lassalle-prim}.
\bigskip

\emph{The edge $E$ is straight (hence $\weight_{M,E}=1$) and
no endpoint of $E$ has degree $1$.} Then
\[ \contribution_{M,\prec}(\lambda) = 
\contribution_{M\setminus E, \prec'}(\lambda). \]
By definition, $E=\{s_1,s_2\}$ being straight means that both of its edge-sides $s_1$ and $s_2$
belong to the same polygon $F \in \loops(\BC,\WC)$.
Besides, there is an even number of edge-sides, let say $2i$ (with $i>0$),
between $s_1$ and $s_2$ if we turn around $F$ in one direction
and also an even number of edge-sides, let say $2j$ (with $j>0$),
if we turn around the face in the other direction.

Fix such a pair $E$ of edge-sides.
Then $(\BC_E,\WC_E)$ is a couple of pairings of $S \setminus E$ of type
$\pi_{\uparrow(i,j)}$ (see \cref{subsec:split}).
As before, the summation over all choices of $(\E,\prec)$
such that $E$ is the first edge is equivalent to a summation over
all choices of the pairing $\E'$ of $S \setminus E$ and all choices of the order $\prec'$ on $\E'$.
By definition, this corresponds to a summation over all choices of $(M',\prec')$,
where $M'=(\BC_E,\WC_E,\E')=M \setminus E$ runs over maps
of face-type $\pi_{\uparrow(i,j)}$.
Therefore, for a fixed $E$, the total contribution of the corresponding pairs
$(M,\prec)$ is equal to 
\[
\sum_{(M',\prec')}
\frac{(-1)^{\ell(\pi)-1}}{ (|\pi|-1)!}
\contribution_{M',\prec'}(\lambda)\\
= \newCh^{(\alpha)}_{\pi_{ \uparrow(i,j)}}(\lambda).\]

Let us count how many pairs $E$ correspond to a given value
of $i$ and $j$.
First, $s_1$ must be chosen in some face $F$ containing $2r$ edge-sides.
There are $m_r(\pi)$ such faces and $2r$ edge-sides in each of them,
so there are $2 r m_r(\pi)$ possible choices for $s_1$.
Once $s_1$ is fixed, there are two possible choices for $s_2$ 
(and only one choice if $i=j$):
we fix arbitrarily a direction to turn around the face $F$
and then $s_2$ must be
the $i+1$-th or $j+1$-th edge-side after $s_1$ in this direction.
As $s_1$ and $s_2$ play identical role and $E$ is a non-ordered pair,
the number of pairs $E$ corresponding to a pair of values $\{i,j\}$
is equal to $(2 - \delta_{i,j}) r m_r(\pi)$.
Hence the total contribution of the couples $(M,\prec)$ such that $E$
is straight and none of its endpoints is a leaf is equal to
\[\sum_{\substack{r \geq 1 \\ \{i,j\} : \\ i,j\geq 1, \\i+j=r-1}} (2 - \delta_{i,j}) r m_r(\pi)
\newCh^{(\alpha)}_{\pi_{\uparrow(i,j)}}(\lambda)
= \sum_{\substack{r \geq 1 \\ i,j\geq 1 \\ i+j =r-1}} r m_r(\pi)
\newCh^{(\alpha)}_{\pi_{\uparrow(i,j)}}(\lambda).\]
Clearly, it is equal to the second summand on the left-hand side of \eqref{eq:recurrence-lassalle-prim}.
\bigskip

\emph{The edge $E$ is twisted and thus $\weight_{M,E}=\gamma$.}
Then, no endpoint of
$E$ has degree $1$, hence
\[ \contribution_{M,\prec}(\lambda) = \gamma
\contribution_{M\setminus E, \prec'}(\lambda).\]
We fix
a pair $E=\{s_1,s_2\}$
such that both edge sides $s_1$ and $s_2$ lie in a polygon $F$
of $\loops(\BC,\WC)$ and are in an odd position.
As above, if we fix the number $2r$ of  the edge-sides in $F$,
there are $2r m_r(\pi)$ possible choices for $s_1$.
Once $s_1$ is fixed, there are $r-1$ possible choices for $s_2$,
which makes $r(r-1)m_r(\pi)$ choices for the pair $\{s_1,s_2\}$
(beware of the symmetry between $s_1$ and $s_2$).

Fix such an edge $E$.
The couple $(\BC_E,\WC_E)$ of pairings of $S \setminus E$ has type
$\pi\downarrow (r)$ (see \cref{subsec:twist}).
Hence the summation over all choices of $(M,\prec)$
such that $E$ is the first edge
is equivalent to
a summation over maps $M \setminus E$ of face-type $\pi\downarrow (r)$.

Finally, the total contribution of the couples $(M,\prec)$
with the first edge twisted is equal to
\begin{multline*} \sum_r r(r-1)\ m_r(\pi) \sum_{(M',\prec')}
\frac{(-1)^{\ell(\pi)-1}}{ (|\pi|-1)!}  \gamma \contribution_{M',\prec'}(\lambda) \\
= - \gamma \sum_r r(r-1)\ m_r(\pi)\ \newCh_{\pi\downarrow (r)}(\lambda), 
\end{multline*}
where the summation on the left-hand side is over maps $M'$ with face-type $\pi\downarrow(r)$. 
Clearly, it is equal to the third summand on the left-hand side of \eqref{eq:recurrence-lassalle-prim}.
\bigskip

\emph{The edge $E$ is interface and thus $\weight_{M,E}=\frac{1}{2}$.}
Then, no
endpoint of
$E$ has degree $1$, hence
\[
\contribution_{M,\prec}(\lambda) = 
\frac{1}{2} \contribution_{M\setminus
E, \prec'}(\lambda).
\]
Fix a pair $E=\{s_1,s_2\}$ of edge-sides lying in two different polygons $F_1$
and $F_2$ of $\loops(\BC,\WC)$.
Suppose $F_1$ contains $2r$ edge-sides, while $F_2$ has $2s$.
Then $(\BC_E,\WC_E)$ has face-type $\pi\downarrow(rs)$ 
(see \cref{subsec:join}).
The summation over all choices of $(M,\prec)$
such that $E$ is the first edge
is equivalent to 
a summation over all choices of the map $M'=M \setminus E$ of face-type $\pi\downarrow (rs)$.
Therefore, for a fixed pair $E$ as above, the total
contribution of the couples $(M,\prec)$ with the first edge equal to $E$ is given by
\[\sum_{(M',\prec')}     
\frac{(-1)^{\ell(\pi)-1}}{ (|\pi|-1)!} \cdot 
\frac{1}{2} \contribution_{M',\prec'}(\lambda) =
\frac{1}{2} \newCh_{\pi\downarrow (rs)}(\lambda).\]

How many pairs $E$ correspond to a given pair $\{r,s\}$?
First, one should choose $s_1$ in a polygon of size $2r$ or $2s$, let us say
$2r$, of $\loops(S_1,S_2)$.
There are $2r m_r(\pi)$ choices for that.
Then we choose $s_2$ in a polygon of size $2s$ of $\loops(S_1,S_2)$
(beware that if $r=s$, this polygon has to be different from the first one):
there are $2s \big(m_s(\pi) -\delta_{r,s}\big)$ choices for that.
If $r=s$ then $s_1$ and $s_2$ play analogous roles 
(if $r \neq s$, we broke the symmetry
by assuming that $s_1$ lies in a polygon of size $2r$),
so one should divide by $2$ in order to count unordered pairs $\{s_1,s_2\}$ instead of ordered pairs.
Finally, we get that the total contribution of the couples $(M,\prec)$
with the first edge being interface is equal to
\begin{multline*}
    \sum_{\{r,s\} }\frac{4}{1+\delta_{r,s}}
rs\ m_r(\pi) \big( m_s(\pi)- \delta_{r,s} \big)\
\frac{1}{2} \newCh_{\pi\downarrow (rs)}(\lambda)\\
=\sum_{r,s }       
rs\ m_r(\pi) \big( m_s(\pi)- \delta_{r,s} \big)\
\newCh_{\pi\downarrow (rs)}(\lambda).
\end{multline*}

Clearly, it is equal to the fourth summand on the left-hand side of \eqref{eq:recurrence-lassalle-prim}.
\bigskip

By bringing together all the contributions of the cases considered above we finish the proof.
\end{proof}

\subsection{Proof of \cref{thm:rectangular-ok}}

\begin{proof}[Proof of \cref{thm:rectangular-ok}]
We will use induction over $|\pi|$. 
For $|\pi|=0$ there is only the empty partition $\pi=\emptyset$; 
clearly in this case $\Ch^{(\alpha)}_{\emptyset}(\lambda) = \newCh^{(\alpha)}_{\emptyset}(\lambda) =1$ holds true. 
Since this is a bit pathological case (empty polygon, empty function, etc.), 
in order to avoid difficulties with the start of the induction, 
we also consider separately the case $|\pi|=1$ for which there is only one partition
$\pi=(1)$; we easily get that
that $\Ch^{(\alpha)}_{1}(\lambda) = \newCh^{(\alpha)}_{1}(\lambda) =
|\lambda|$ indeed holds true.

Let us assume that the inductive assertion holds for all $\pi$ such that $|\pi| < n$ and 
let $\pi$ be a partition with $|\pi|=n$.
In the case when $m_1(\pi)\geq 1$
we apply \eqref{eq:nielubie1A} and \eqref{eq:nielubie1B}
and the inductive assertion implies that \eqref{eq:equality-rectangular} holds true for $\pi$ as well.

In the case when $m_1(\pi)=0$, we compare the left-hand side of \eqref{eq:recurrence-lassalle-prim} 
with the left-hand side of \eqref{eq:recurrence-lassalle}. 
From the inductive assertion it follows that they are equal; so must be their right-hand sides.
This concludes the proof of the inductive step.
\end{proof}

\subsection{Proof of \cref{thm:ogs-ok-for-rectangle}}
\label{sec:proof-OGS-rectangle}

\newcommand{\isotropic}[1]{{#1'}}

\begin{proof}[Proof of \cref{thm:ogs-ok-for-rectangle}] 
The first part of \cref{thm:ogs-ok-for-rectangle} is restated in \cref{thm:rectangular-ok}, 
hence it is enough to prove only the second part.

In the following we shall implicitly view $\alpha$ as a function of $\gamma$, where $\alpha(\gamma)$ 
was defined in the beginning of the proof of \cref{lem:stanley-polynomials-exist}.

Let us fix the values of an integer $\ell\geq 1$ and $1 \leq i \leq \ell$.
We define
\begin{multline*}
    \widetilde{PQ}_i := \Bigg\{\big(\widetilde{P}_i = (p_1,\dots,p_\ell),\ Q = (q_1,\dots,q_\ell),\ \gamma \big) : \\
\text{$p_j = 0$ for all $j \neq i$ and} \\
\sqrt{\alpha} \widetilde{P}_i \times \frac{1}{\sqrt{\alpha}} Q 
\text{ is a Young diagram}\Bigg\}. 
\end{multline*}
Note that all Young diagrams appearing in this set are rectangular.
Thus the first part of \cref{thm:ogs-ok-for-rectangle} implies that 
for any triple $(\widetilde{P}_i,Q,\gamma) \in \widetilde{PQ}_i$ the equality
\begin{equation}
\label{eq:InProof}
\Ch^{(\alpha)}_\pi \left(\sqrt{\alpha} \widetilde{P}_i \times \frac{1}{\sqrt{\alpha}} Q\right) = 
\newCh^{(\alpha)}_\pi \left(\sqrt{\alpha} \widetilde{P}_i \times \frac{1}{\sqrt{\alpha}} Q\right)
\end{equation}
holds true for any partition $\pi$. 

Just like in \cref{def:ogs-property}, we shall view now $P=(p_1,\dots,p_\ell)$ and $Q=(q_1,\dots,q_\ell)$ 
as sequences of indeterminates. Then, by \cref{lem:stanley-polynomials-exist} we know that each of the quantities
\[ \Ch^{(\alpha)}_\pi \left( \sqrt{\alpha}P\times \frac{1}{\sqrt{\alpha}}Q \right) \text{ and } 
\newCh^{(\alpha)}_\pi \left( \sqrt{\alpha}P\times \frac{1}{\sqrt{\alpha}}Q \right)\]
is a polynomial in the indeterminates $\gamma, p_1,\dots,p_\ell,q_1,\dots,q_\ell$. 
By substituting
\[\widetilde{P}_i := (\underbrace{0,\dots,0}_{\text{$i-1$ times}},p_i,\underbrace{0,\dots,0}_{\text{$\ell-i$ times}})\] 
we have that each of the quantities:
\begin{equation}
\label{eq:polynomial-widetilde-A}
\Ch^{(\alpha)}_\pi \left( \sqrt{\alpha}\widetilde{P}_i\times \frac{1}{\sqrt{\alpha}}Q \right)   
\end{equation}
and
\begin{equation}
\label{eq:polynomial-widetilde-B}
\newCh^{(\alpha)}_\pi \left( \sqrt{\alpha}\widetilde{P}_i\times \frac{1}{\sqrt{\alpha}}Q \right)   
\end{equation}
can be expressed as a polynomial in the indeterminates $\gamma, p_i,q_1,\dots,q_\ell$ \linebreak
(\emph{a~priori} the uniqueness of this polynomial might be not obvious; we shall discuss this issue in the following).

Moreover, for any triple $(g,e,f)$ of non-negative integers, 
the following equality between the coefficients of the respective polynomials holds true:
\begin{align*}
\left[\gamma^g p_i^e q_i^f\right] \Ch^{(\alpha)}_\pi \left( \sqrt{\alpha}P\times \frac{1}{\sqrt{\alpha}}Q \right) 
&=
\left[\gamma^g p_i^e q_i^f\right]
\Ch^{(\alpha)}_\pi\left( \sqrt{\alpha}\widetilde{P}_i\times \frac{1}{\sqrt{\alpha}}Q \right) , \\
\intertext{and similarly}
\left[\gamma^g p_i^e q_i^f\right] \newCh^{(\alpha)}_\pi\left(\sqrt{\alpha}P\times \frac{1}{\sqrt{\alpha}}Q\right) 
&=
\left[\gamma^g p_i^e q_i^f\right]
\newCh^{(\alpha)}_\pi\left( \sqrt{\alpha}\widetilde{P}_i\times \frac{1}{\sqrt{\alpha}}Q \right) . 
\end{align*}

Since equality \eqref{eq:InProof} holds for any triple
$(\widetilde{P}_i,Q,\gamma) \in \widetilde{PQ}_i$, 
it follows --- using the same technique as in the uniqueness part of the proof of \cref{lem:stanley-polynomials-exist} 
--- that the quantities
\[ \Ch^{(\alpha)}_\pi \left( \sqrt{\alpha}\widetilde{P}_i\times \frac{1}{\sqrt{\alpha}}Q \right) \text{ and }
 \newCh^{(\alpha)}_\pi \left( \sqrt{\alpha}\widetilde{P}_i\times \frac{1}{\sqrt{\alpha}}Q \right)\] 
are equal as polynomials in the indeterminates $\gamma, p_i,q_1,\dots,q_\ell$
(note that the same argument explains the uniqueness of the polynomials \eqref{eq:polynomial-widetilde-A}
and \eqref{eq:polynomial-widetilde-B}). 
In particular, for any triple $(g,e,f)$ of non-negative integers, 
the following equality between the coefficients of the respective polynomials in the indeterminates 
$\gamma, p_1,\dots,p_\ell,q_1,\dots,q_\ell$ holds true:
\begin{align*}
\left[\gamma^g p_i^e q_i^f\right] \Ch^{(\alpha)}_\pi \left( \sqrt{\alpha}P\times \frac{1}{\sqrt{\alpha}}Q \right) 
&=
\left[\gamma^g p_i^e q_i^f\right] \newCh^{(\alpha)}_\pi\left(\sqrt{\alpha}P\times \frac{1}{\sqrt{\alpha}}Q\right),
\end{align*}
which finishes the proof.
\end{proof}

\section{Support for the conjectures: special values of $\alpha$}
\label{sec:special_alpha}

\subsection{Reformulation of the results from \cite{FeraySniady2011}}
\label{SubsectReformulationFS11}
The purpose of this paragraph is to explain how
Equations~\eqref{eq:alpha=2} and \eqref{eq:alpha=1/2} can be obtained
easily from the results of \cite{FeraySniady2011}, even if the presentation
there is a little bit different.
We use boldface characters for the notation of \cite{FeraySniady2011}.

Firstly,  note the difference in the notation and in the normalization:
\[\Ch_\pi^{(2)}(\lambda) = 
\left(\frac{1}{\sqrt{2}}\right)^{|\pi|-\ell(\pi)} \bm{\varSigma_\pi^{(2)}}. \]
Note also that the roles of black and white vertices in the definition
of $N_G$ are inverted.
Hence \cite[Theorem 5.2]{FeraySniady2011} with the notation 
of the present paper takes the form
\[\Ch_\pi^{(2)}(\lambda)=\left(\frac{1}{\sqrt{2}}\right)^{|\pi|-\ell(\pi)}
\frac{(-1)^{|\pi|}}{2^{\ell(\pi)}} \sum_M (-2)^{|\V_\circ(M)|}\ N^{(1)}_M,\]
where the sum runs over the maps of face-type $\pi$.
This is clearly the same as Equation \eqref{eq:alpha=2}.

Equation \eqref{eq:alpha=1/2} is deduced directly
from \eqref{eq:alpha=2} using the duality relation 
\eqref{EqDuality} and the fact that 
\[N^{(\alpha)}_G(\lambda')=N^{(1/\alpha)}_{G'}(\lambda),\]
where $G'$ is obtained from $G$ by inverting the colors of the vertices.

\subsection{The special cases $\alpha=\frac{1}{2}$ and $\alpha=2$}

\begin{theorem}
\label{thm:special-values-are-ok}
The answer for \cref{conj:precise} is positive for $\alpha=\frac{1}{2}$ and $\alpha=2$.
\end{theorem}
\begin{proof}
The condition $\alpha \in \left\{\frac{1}{2},2\right\}$ is
equivalent to $\gamma^2=\frac{1}{2}$.
This implies, by the same case analysis as in the proof of \cref{lem:how-twisted},
that for any edge $E$ of an arbitrary map $M$
\[ \weight_{M,E} = \gamma^{\big[|\F(M)|-|\V(M)|\big]-
\big[|\F(M\setminus E)|-|\V(M\setminus E)|\big]+1}. \]
For an arbitrary history $\prec$, 
the exponents of $\gamma$ in the product \eqref{eq:weight-map-order} form a telescopic sum, 
thus
\[ \weight_{M,\prec} = \gamma^{|\F(M)|-|\V(M)|+|\E(M)|}=
\gamma^{\ell(\pi)-|\V(M)|+|\pi|},\]
where $\pi$ is the face-type of $M$;
in particular this expression does not depend on the choice of the history,
hence
\[ \weight_{M} = \gamma^{\ell(\pi)-|\V(M)|+|\pi|}.\]

One can thus easily check that $\newCh_\pi^{(2)}$, respectively $\newCh_\pi^{(1/2)}$
coincides with \eqref{eq:alpha=2}, respectively \eqref{eq:alpha=1/2}.
\end{proof}

\section{Link with a positivity conjecture of Lassalle}
\label{sec:link_Lassalle}

In this section we will often use the following parameter introduced by Lassalle:
\[\beta := \alpha - 1.\]

\subsection{Statement of Lassalle's conjecture}

Lassalle stated the following conjecture.
\begin{conjecture}[{\cite[Conjecture 1]{Lassalle2008a}}]
\label{conj:lassalle-stanley}
Let $\pi$ be a partition such that $m_1(\pi) = 0$. Then
$(-1)^{|\pi|}\bm{\vartheta}_{\pi \cup 1^{n-|\pi|}}^{P \times Q}(\alpha)$ 
is a polynomial in $(P, -Q, \beta)$ with non-negative integer coefficients.
\end{conjecture}

In the following we will prove (in \cref{cor:lassalle-for-stanley})
that our Main Conjecture \ref{conj:main} implies
a weaker version of \cref{conj:lassalle-stanley},
namely that the coefficients are non-negative \emph{rational} numbers.

\subsection{Positivity in multirectangular coordinates}

Our first step is the following statement.
 \begin{theorem}
 \label{theo:conseq1}
 Let us assume that  
 \cref{conj:main} holds true. Then $(-1)^{|\pi|}\Ch^{(\alpha)}_\pi(P \times Q)$ is a polynomial in
 the variables $(P/\sqrt{\alpha},-\sqrt{\alpha}Q,-\gamma)$ with non-negative
 rational coefficients.
 \end{theorem}
 
\begin{proof}
For a Young diagram $P\times Q$ given in the multirectangular coordinates, 
the number of  embeddings $N^{(1)}_M(P\times Q)$
takes a particularly simple form (see \cite[Lemma 3.9]{FeraySniady2011} where the notations are slightly different), 
for this reason, \cref{eq:conjectural-mysterious-formula} would imply that
\begin{multline*}
\Ch^{(\alpha)}_\pi(P \times Q) = (-1)^{\ell(\pi)} \sum_{M}\
\vagueweight_M(\gamma)\ \times \\
\left[
\sum_{\varphi\colon \V_\bullet(M) \to \N^\star}\
\prod_{l \in \V_\bullet(M)}\left( \frac{-p_{\varphi(l)}}{\sqrt{\alpha}}\right) \cdot
\prod_{l' \in \V_\circ(M)}\left(\sqrt{\alpha}\ q_{\psi(l')}\right)
\right], 
\end{multline*}
where the first sum is over all bicolored maps $M$ of the face-type $\pi$ and
$\psi(l')$ is defined as the maximum of $\varphi(l)$ over all white neighbors $l$ of the black vertex $l'$. 

The quantity $\vagueweight_M(\gamma)$ is a polynomial in $\gamma$ with non-negative rational
 coefficients of the same parity as the Euler characteristic $\chi(M)$,
 that is the same parity as $|\pi|+\ell(\pi)+|\V(M)|$.
  Rewriting the equation above as
\begin{multline*}
    (-1)^{|\pi|} \Ch^{(\alpha)}_\pi(P \times Q) = \sum_{M}\
\vagueweight_M(-\gamma)\ \times \\
\left[
\sum_{\varphi\colon \V_\bullet(M) \to \N^\star}\
\prod_{l \in \V_\bullet(M)}\left( \frac{p_{\varphi(l)}}{\sqrt{\alpha}}\right) \cdot
\prod_{l' \in \V_\circ(M)}\left(-\sqrt{\alpha}\ q_{\psi(l')}\right)
\right]
\end{multline*}
finishes the proof.
\end{proof}
 
 \begin{corollary}
\label{cor:lassalle-for-stanley}
Let us assume that 
\cref{conj:main} holds true.
Let $\pi$ be an arbitrary partition. Then
$(-1)^{|\pi|}\bm{\vartheta}_{\pi \cup 1^{n-|\pi|}}^{P \times Q}(\alpha)$ is a polynomial in $(P, -Q, \beta)$ with non-negative rational coefficients. 
\end{corollary}
 
 \begin{proof}
Using \eqref{eq:lassalle-normalization-character}
we obtain:
 \begin{multline*}
 (-1)^{|\pi|}\bm{\vartheta}_{\pi \cup 1^{n-|\pi|}}^{P \times Q}(\alpha) =  
\sum_{M}\ 
\sqrt{\alpha}^{2|\V_\circ(M)|+|\pi|-\ell(\pi)  - |\V(M)|}
\vagueweight_M(-\gamma)\\
  \times \left[
 \sum_{\varphi\colon \V_\bullet(M) \to \N^\star}\
 \prod_{l \in \V_\bullet(M)} p_{\varphi(l)} \cdot
 \prod_{l' \in \V_\circ(M)}\left(- q_{\psi(l')}\right)
 \right].
 \end{multline*}
Recall that
$\vagueweight_M(\gamma)$ is a polynomial in $\gamma$ of degree at most 
 \[\chi(M)= 2 (\text{number of connected components of $M$}) - \ell(\pi) + |\pi| - |\V(M)|\] 
and with the same parity as $\chi(M)$.
The number of connected components of $M$ is at most equal to the number of white vertices, and since $-\gamma = \frac{\beta}{\sqrt{\alpha}}$ and $\alpha= \beta+1$ we have that
 \[ \sqrt{\alpha}^{2|\V_\circ(M)|+|\pi|-\ell(\pi)  - |\V(M)|}
 \vagueweight_M(-\gamma) \]
 is a polynomial in $\beta$ with non-negative rational coefficients, which concludes the proof.
\end{proof}

\begin{remark}
    If \cref{conj:main} is true with a weight $\vagueweight_M$ with integer coefficients 
    (as polynomial in $\gamma$),
    then it also implies the integrality statements in \cite[Conjecture 1]{Lassalle2008a}
    and \cite[Conjecture 1.2]{Lassalle2009}.
\end{remark}

\section{Computer exploration and the counterexample}
\label{sec:Conj_Is_False}

\subsection{Counterexample $\pi=(9)$}

For $\pi=(9)$ a computer calculation shows that
\begin{multline}
\label{eq:counterexample}
\newCh^{(\alpha)}_{(9)}(P\times Q)-
\Ch^{(\alpha)}_{(9)}(P\times Q)= \\
\frac{41}{70} (2 \gamma^2-1) \sum_{i<j<k} p_{i} p_{j} p_{k} (q_k-q_j)(q_i-q_j) q_{k}
\end{multline}
which might be non-zero for multirectangular Young diagrams consisting of at least $\ell\geq 3$ rectangles.
It is worth pointing out that this is \emph{not} a counterexample for \cref{conj:2-rectangle}. 
However, it shows that the answer to \cref{conj:precise} might be negative for some specific choices of
$\alpha$ and $\lambda$.

For \eqref{eq:counterexample}, 
the quantity $\newCh^{(\alpha)}_\pi(P\times Q)$ was computed using the
very definition given in this article.
Computing $\Ch^{(\alpha)}_\pi(P \times Q)$ is a bit harder (while shorter in practice):
we used some data made available by Lassalle \cite{LassalleData},
that express it in terms of the free cumulants
(Lassalle gave an algorithm to do this computation \cite[Section 9]{Lassalle2009},
but as his data was made available, we did not implement it again).
Then the free cumulant $R_k(P \times Q)$ can be computed
using the recursive structure of bicolored planted plane trees 
(see \cite[Equations (10), (11) and (12)]{Rattan2007}).

The calculation of \eqref{eq:counterexample} took a week of computer time.
Finding this counterexample was only possible because the theoretical results
in this paper
and some additional tricks allow to reduce the computational complexity
(the naive algorithm which lists all maps with all histories would have to consider 
$17!! \cdot 9!\approx 1.25\times 10^{13}$ cases).
An analogous calculation for $\pi=(10)$ would be, for the moment, rather challenging.

\subsection{Another weight}
We have also been testing numerically
another candidate for the weight in \cref{conj:main}.
The idea was to define $\weight'_M$ as $\weight_{M,<}$ for some specific
history $<$.
We chose this history as follows:
\begin{itemize}
    \item first erase the edge containing the edge-side with
        the smallest label (denote it $s_0$);
    \item then remove the edges containing 
        \[\E \circ \WC (s_0), (\E \circ \WC)^2(s_0), \cdots \]
        until you reach $s_0$ again 
        (here the pairings $\E$ and $\WC$ are viewed as fixpoint-free involutions);
    \item then start again with the edge-side with the smallest label among the
        remaining edges.
\end{itemize}
This definition was inspired by the work of La Croix on $b$-Conjecture
\cite[Section 4.1]{La2009}.
This new candidate weight $\weight'_M$ is much easier to evaluate as we do not need to consider
all possible histories.

This new weight $\weight'_M$ gives the correct answer in the cases 
$\alpha\in\left\{1/2,2\right\}$, just as $\weight_M$ does (\cref{thm:special-values-are-ok}).
Indeed, we have proved that in this case $\weight_{M,\prec}$
does not depend on the choice of the history $\prec$, thus
any specific choice of history (or any mean over some set of histories)
works fine.

We have observed numerically that this weight is a solution to 
\cref{conj:main} for any $\pi$ of size at most $8$ as well as for $\pi=(9)$,
but \textbf{not} for $\pi=(10)$ and $\pi=(5,4)$.

Also, this weight seems to work for rectangular shapes,
but we are unable to prove it.
Numerical data suggests that it also
works for a superposition of (at most) two rectangles 
$\lambda=(p_1,p_2) \times (q_1,q_2)$.

\section*{Acknowledgments}

We thank Michael La Croix for a very interesting discussion concerning $b$-Conjecture.

We also thank Michel Lassalle for 
making his data available on his web-page \cite{LassalleData}.
Computer exploration for this paper was partly driven with
the open-source
mathematical software \texttt{Sage}~\cite{sage} and its algebraic
combinatorics features developed by the \texttt{Sage-Combinat}
community~\cite{Sage-Combinat}.

M.D.'s research has been supported by a grant of \emph{Narodowe Centrum Nauki} (2011/03/N/ST1/00117).
In the initial phase of research, P.\'S.~was a holder of a fellowship of \emph{Alexander von Humboldt-Stiftung}.
P.\'S.'s research has been supported by a grant of \emph{Deutsche Forschungsgemeinschaft} (SN 101/1-1).
V.F.'s research is partially supported by \emph{ANR} grant \textsc{psyco} (ANR-11-JS02-001).


\newcommand{\etalchar}[1]{$^{#1}$}
\def\cprime{$'$}

\end{document}